\documentclass[11pt]{article}
\usepackage{amsmath,amsthm,amsfonts,amssymb,bbm}
\usepackage[sort, numbers, square]{natbib}
\usepackage{enumerate}
\usepackage{multirow} 
\usepackage{graphicx}
\usepackage[all]{xy}
\usepackage{tabularx} 
\usepackage{booktabs} 
\usepackage{accents} 
\usepackage{xfrac} 
\usepackage[font=small,labelfont=bf]{caption} 
\usepackage{rotate}

\usepackage{titlesec}

\titleformat*{\section}{\large\bfseries}
\titleformat*{\subsection}{\large\bfseries}
\titleformat*{\subsubsection}{\large\bfseries}

\newcommand{\cf}{cf.\ }

\newcommand{\resp}{resp.\ }

\newcommand{\ie}{i.e.\ } 
\newcommand{\eg}{e.g.\ } 
\newcommand{\iid}{i.i.d.\ } 
\newcommand{\fdd}{f.d.d.\ } 
\newcommand{\wrt}{w.r.t.\ } 
\newcommand{\rhs}{r.h.s.\ }

\newcommand{\RR}{\ensuremath{\mathbb{R}}}

\newcommand{\NN}{\ensuremath{\mathbb{N}}}
\newcommand{\ZZ}{\ensuremath{\mathbb{Z}}}
\newcommand{\QQ}{\ensuremath{\mathbb{Q}}}
\newcommand{\PP}{\ensuremath{\mathbb{P}}}
\newcommand{\EE}{\ensuremath{\mathbb{E}}}

\DeclareMathOperator{\conv}{conv}

\newcommand{\pr}{\ensuremath{\text{\normalfont pr}}}

\newcommand{\Ex}{\ensuremath{\text{\normalfont Ex}}}
\newcommand{\tm}{\ensuremath{\text{\normalfont TM}}}
\newcommand{\tcf}{\ensuremath{\text{\normalfont TCF}}}
\newcommand{\hyp}{\ensuremath{\text{\normalfont HYP}}}
\newcommand{\psd}{\ensuremath{\text{\normalfont PSD}}}

\newcommand{\maxstab}{\ensuremath{\text{\normalfont MAX}}}
\newcommand{\bin}{\ensuremath{\text{\normalfont BIN}}}
\newcommand{\cpp}{\ensuremath{\text{\normalfont CPP}}}
\newcommand{\cut}{\ensuremath{\text{\normalfont CUT}}^{\square}}
\newcommand{\cor}{\ensuremath{\text{\normalfont COR}}^{\square}}
\newcommand{\hy}{\ensuremath{\text{\normalfont hyp}}}
\newcommand{\pu}{\ensuremath{\text{\normalfont pure}}}

\newcommand{\A}{\ensuremath{{\mathcal A}}}
\newcommand{\finite}{\ensuremath{\mathcal{F}}}

\newcommand{\eins}{\ensuremath{\mathbf{1}}}
\newcommand{\Eins}{\ensuremath{\mathbbm{1}}}

\newcommand{\pmid}{\ensuremath{\,:\,}}

\newcommand{\Max}{\ensuremath{\bigvee}}
\newcommand{\Min}{\ensuremath{\bigwedge}}

\theoremstyle{plain}
\newtheorem{theorem}{Theorem}
\newtheorem{proposition}[theorem]{Proposition}
\newtheorem{corollary}[theorem]{Corollary}
\newtheorem{lemma}[theorem]{Lemma}

\theoremstyle{definition}

\newtheorem{example}[theorem]{Example}

\theoremstyle{remark}
\newtheorem{remark}[theorem]{Remark}

\def\OH{$\sfrac{1}{2}$} 
\def\OT{$\sfrac{1}{3}$} 
\def\TT{$\sfrac{2}{3}$} 
\def\ol{$\times$} 

\begin{document}

\title{The realization problem for tail correlation functions} 

\author{
  Ulf-Rainer Fiebig\footnote{\mbox{Institute for Mathematical Stochastics, University of G\"ottingen, Email: urfiebig@math.uni-goettingen.de}}, \, 
  Kirstin Strokorb\footnote{Institute of Mathematics, University of Mannheim, Email: strokorb@math.uni-mannheim.de} \, and 
  Martin Schlather\footnote{Institute of Mathematics, University of Mannheim, Email: schlather@math.uni-mannheim.de}
}

\maketitle 

\begin{abstract}
For a stochastic process $\{X_t\}_{t \in T}$ with identical one-dimensional margins and upper endpoint $\tau_{\text{up}}$ its tail correlation function (TCF) is defined through $\chi^{(X)}(s,t) =  \lim_{\tau \to \tau_{\text{up}}} P(X_s > \tau \,\mid\,   X_t > \tau )$. It is a popular bivariate summary measure that has been frequently used in the literature in order to assess tail dependence. In this article, we study its realization problem.
We show that the set of all TCFs on $T \times T$ coincides with the set of TCFs stemming from a subclass of max-stable processes and can be completely characterized by a system of affine inequalities. Basic closure properties of the set of TCFs and regularity implications of the continuity of $\chi$ are derived. If $T$ is finite, the set of TCFs on $T \times T$ forms a convex polytope of $\lvert T \rvert \times \lvert T \rvert$ matrices.
Several general results reveal its complex geometric structure.
Up to $\lvert T \rvert = 6$ a reduced system of necessary and sufficient conditions for being a TCF is determined. 
None of these conditions will become obsolete as $\lvert T \rvert\geq 3$ grows. 
\end{abstract}

{\small
}


\section*{Introduction}

The study of the existence of stochastic models with some prescribed distributional properties has a long tradition in the theory of probability and various fields of application. Let $\{X_t\}_{t \in T}$ be a stochastic process on some index set $T$ (which may be finite or infinite with some topological structure). Typically, a real-valued summary statistic $\kappa^{(X)}(s,t)$ of the distribution of $(X_s,X_t)$ is of particular interest for all pairs $(s,t) \in T \times T$.
The question is whether for some prescribed function $\kappa$ on $T \times T$ a stochastic model $\{X_t\}_{t \in T}$ exists that \emph{realizes} $\kappa$, i.e. if $\kappa^{(X)}=\kappa$. 
Recent accounts and surveys on such \emph{realization problems} with an emphasis on $\{0,1\}$-valued processes 
(or random sets, two-phased media, binary processes) 
include \cite{torquato02}, \cite{quint08}, \cite{emery10}, \cite{lm15} and \cite{la15}.
Also from a statistical point of view realization problems are important, namely for consistent inference.

As pointed out by \cite{lm15}, the question of realizability usually leads to a (possibly infinite and even in finite setups huge) set of positivity conditions for the quantity of interest, and secondly, to a set of regularity conditions if the topology of the underlying space is of interest as well.
These positivity conditions are needed in statistical applications to correct estimators $\widehat{\kappa}^{(X)}$ for $\kappa^{(X)}$ to an admissible function.

Let us consider a classical example. Assuming that the second moments 
of a real-valued stochastic process $\{X_t\}_{t \in T}$ exist at each locaction $t \in T$, the process possesses a covariance function $C^{(X)}(s,t)=\text{Cov}(X_s,X_t)$. It is well-known that $C=C^{(X)}$ must be \emph{positive semi-definite}, i.e. $C(s,t)=C(t,s)$ and
\begin{align}\label{eq:psd}
\sum_{i=1}^m \sum_{j=1}^m a_ia_j C(t_i,t_j) \geq 0 \quad \forall \, (t_1,\dots,t_m) \in T^m, \, (a_1,\dots,a_m) \in \RR^m, \, m \in \NN. 
\end{align}
Conversely, for any such function $C$, there exists a stochastic process $\{X_t\}_{t \in T}$ with covariance function $C^{(X)}=C$. The stochastic process $\{X_t\}_{t \in T}$ is not unique, but it may be chosen to be a centered Gaussian process as can be easily checked by  Kolmogorov's extension theorem.  Such a process on the space $T$ (if additionally equipped with some topology), may have very uncomfortable regularity properties. Several authors have  established connections between the regularity of the covariance function $C$ and the existence of a corresponding stochastic process with a certain sample path regularity, \cf e.g.\ \cite{adler90} for an overview in case of continuity. 
In statistical applications, the development of efficient non-parametric estimators for the covariance function that ensure positive semi-definiteness can be a challenging task, cf.\  e.g.\ \cite{halletal} and \cite{politis11}.

When it comes to the extreme values in the upper quantile regions of a real-valued stochastic process $\{X_t\}_{t \in T}$, summary measures like the covariance function often do not exist and would be genuinely inappropriate to characterize dependence.
Instead, among several other summary statistics that have emerged in an extreme value context (cf.\ for instance \cite{beirlantetal04} Section 8.2.7), the following bivariate quantity
\begin{align*}
\chi^{(X)}(s,t):=\lim_{\tau \to \tau_{\text{up}}} \PP(X_s > \tau \,|\, X_t > \tau), \qquad s,t \in T,
\end{align*}
which we call \emph{tail correlation function} (TCF) \citep{strokorbballanischlather_15}, 
has received particular attention.
As commonly done and in accordance with stationarity assumptions, we assume here and hereafter that $\{X_t\}_{t \in T}$ has identical one-dimensional marginal distributions with upper endpoint $\tau_{\text{up}}$ (which may be $\infty$).  

Dating back to \cite{geffroy_5859}, \cite{sibuya_60} and \cite{tiagodeoliveira_62}, the TCF enjoys steady popularity among practitioners and scholars in order to account for tail dependence, albeit frequently reported 
under different names. 
The insurance, finance, economics and risk management literature knows it mainly as \emph{(upper) tail dependence coefficient} \citep{frahmetal05}, \emph{coefficient of (upper) tail dependence} \citep{mcneil2003modelling} or simply as \emph{(upper) tail dependence} \citep{patton06}.
In environmental contexts it has been additionally addressed as \emph{$\chi$-measure} \citep{colesheffernantawn_99}. Spatial environmental applications tend to prefer the equivalent quantity $2-\chi$, referred to as \emph{extremal coefficient function}. Among many others, the references \cite{blanchetdavison_11}, \cite{emks_15} and \cite{thibaudopitz_15} use it as an exploratory tool for testing the goodness of fit. In the context of stationary time series, the TCF constitutes a special case of the \emph{extremogram} \citep{davismikosch_09}. Moreover, the standard classification of the random pair $(X_s,X_t)$ as exhibiting either \emph{asymptotic/extremal independence} (when $\chi^{(X)}(s,t)=0$)  or \emph{asymptotical/extremal dependence} (when $\chi^{(X)}(s,t)\in (0,1]$) is based on the TCF $\chi$.

Even though the TCF is a ubiquitous quantity within the extremes literature, surprisingly little is known about the class of TCFs and even less when it comes to the interplay of TCFs and their realizing models. 
This is the central theme of the present text.
That is, 
we are aiming  at giving at least partial answers to the following questions:

{\it
\begin{enumerate}[(A)]
\item Can we decide if a given real-valued function $\chi: T\times T \rightarrow \RR$ is the TCF of a stochastic process $\{X_t\}_{t \in T}$?
\item If this is the case, can we find a specific stochastic process $\{X_t\}_{t \in T}$ with $\chi^{(X)}=\chi$? 
\end{enumerate}
}

We also address the following regularity question.

{\it
\begin{enumerate}[(C)]
\item Does the continuity of a TCF $\chi$ imply the existence of a stochastic process realizing $\chi$ that additionally satisfies some regularity property?
\end{enumerate}
}

A satisfactory answer to Question (A) is desirable in a statistical context 
in order to decide whether estimators of the TCF produce admissible TCFs as an outcome.
This concerns specifically spatial applications where one is bound to encounter very high-dimensional observations and therefore only partial low-dimensional information (such as the TCF) can be taken into account for inference. A first attempt to include properties of the class of TCFs to improve statistical inference can be found in  \cite{schlathertawn_03}.
The TCF $\chi=\chi^{(X)}$ is a \emph{non-negative correlation function}. That is, $\chi$ is positive semi-definite in the sense of \eqref{eq:psd} with $\chi(s,t) \geq 0$ and $\chi(t,t)=1$ for all $s,t \in T$ (cf.\ e.g.\ \cite{schlathertawn_03}, \cite{davismikosch_09} and \cite{fasenetalii_10}).
However, even though TCFs are non-negative correlation functions, not all such functions  are TCFs. For instance, $\eta:=1-\chi$ has to satisfy the \emph{triangle inequality}
\begin{align}\label{eqn:triangle}
\eta(s,t) \leq \eta(s,r) + \eta(r,t) \qquad r,s,t \in T
\end{align}
\citep{schlathertawn_03}. In the context of $\{0,1\}$-valued stochastic processes, it is well-known that the respective covariance functions obey this triangle inequality and implications are addressed \eg in \cite{matheron_87}, \cite{markov_95} and \cite{jiaoetalii_07}. If $T=\RR^d$ and the underlying process is stationary, then the function $h \mapsto \chi(o,h)$ (with $o \in \RR^d$ being the origin) cannot be differentiable unless it is constant.

The simplest TCFs are the constant function $\chi(s,t) = 1$ realized by a process of identical random variables, and the function $\chi(s,t)=\delta_{st}:=\Eins_{s=t}$ realized by a process of independent random variables.
Another example for $\chi^{(X)}(s,t)=\delta_{st}$ is a {Gaussian process} $X$ on $T$, whose correlation function $\rho$ on $T \times T$ attains the value $1$ only on the diagonal $\{(t,t) : t \in T\}$ \citep[Theorem~3]{sibuya_60}.
While Gaussian processes do not exhibit tail dependence, the class of \emph{max-stable processes} naturally provides rich classes of non-trivial TCFs.
For instance, any function of the form $\chi(s,t)=\int_{[0,\infty)} \exp\left(-\lambda \lVert s - t \rVert \right) \Lambda(d\lambda)$ will be the TCF of a max-stable process on $T=\RR^d$, if $\Lambda$ is a probability measure on $[0,\infty)$ \citep{strokorbballanischlather_15}. 
Beyond the realizability question, \cite{kabluchkoschlather_10} and \cite{wangroystoev_13} establish some connections between mixing properties of $X$ and decay properties of its TCF $\chi^{(X)}$ when $T=\RR^d$ and $X=\{X_t\}_{t \in \RR^d}$ is stationary and {max-stable}.
It is natural to ask whether even further TCFs will arise if we do not restrict ourselves to the max-stable class, since an affirmative answer would imply a first important reduction for the questions $(A)$-$(C)$.

{\it
\begin{enumerate}[(D)]
\item Is the set of TCFs stemming from max-stable processes  properly contained  in the set of all TCFs or do these sets coincide?
\end{enumerate}
}

Finally, realization problems are usually intimately connected with the question of admissible operations on the quantities of interest. To illustrate this again by means of covariance funcions, note that the product and convex combination of two covariance functions and the pointwise limit of a sequence of covariance functions is again a covariance function. We ask the same question for TCFs.

{\it
\begin{enumerate}[(E)]
\item Is the set of TCFs closed under basic operations such as\\ taking (pointwise) products, convex combinations and limits?
\end{enumerate}
}

In order to deal with the questions above, we establish close connections with $\{0,1\}$-valued processes, polytopes, partitions of sets and combinatorics. 
Recent developments indicate that such tools may appear more frequently in the analysis of extremes, cf.\ \cite{molchanov_08}, \cite{yuenstoev14}, \cite{wangstoev11}, \cite{ehw15} and \cite{thibaudetal15}.

We divide the text into two parts. 

\textbf{Part I} deals with the realization problem of TCFs of stochastic processes $\{X_t\}_{t \in T}$ on arbitrary base spaces $T$. 
Close connections with $\{0,1\}$-valued processes will be established and enter the subsequent considerations. 
We give answers to Questions (D) and (E), partial answers to the Questions (A), (B) and (C) and reduce Question (A) to infinitely (countably) many finite-dimensional problems (in case our base space $T$ is countable). 

\textbf{Part II} deals with these finite-dimensional problems, that is, the realization problem of TCFs of random vectors $\{X_t\}_{t \in T}$ on finite base spaces $T$ with $\lvert T \rvert=n$ for some $n \in \NN$. We are aiming at establishing a (reduced) system of necessary and sufficient conditions for deciding whether a given function is a TCF or not and study the geometry of the set of TCFs. Arguments used in this part will be related to the study of polytopes, are often of combinatorial nature or are based on additional software computations. The latter is a typical phenomenon for realization problems of this kind. 

More detailed descriptions are given at the beginning of each part. Finally, we end with a discussion of our results. 
The appendix contains all tables.\\[4mm]


{\noindent \textbf{\large Part I\\[2mm] The realization problem for TCFs on arbitrary sets $T$}}\\[3mm]


\noindent To start with, Section~\ref{sect:MST_ECF_TM} reviews some preparatory results on max-stable processes, extremal coefficient functions and a particular subclass of max-stable processes, which we called Tawn-Molchanov (TM) processes \citep{strokorbschlather_15}. These processes are important for our analysis, since it turns out that any TCF can be realized by (at least one) TM process, our main result in Section~\ref{sect:realizedTM} and a substantial reduction of the realization problem of TCFs. Section~\ref{sect:realizedTM} also reveals a close connection between the class of TCFs and the class of correlation functions of $\{0,1\}$-valued stochastic processes and addresses the existence of stochastic processes for a prescribed TCF $\chi$ with some minimal regularity properties if $\chi$ is at least continuous. Subsequently,  Section~\ref{sect:props} collects some immediate consequences concerning closure properties of the set of TCFs and the characterization of the set of TCFs by means of finite-dimensional inequalities, our starting point for Part II.

\section{{Max-stable processes, extremal coefficients and TM processes}}\label{sect:MST_ECF_TM}

A stochastic process $X=\{X_t\}_{t \in T}$ is \emph{simple max-stable}, if it has unit Fr{\'e}chet margins (meaning $\PP(X_t \leq x)=\exp(-1/x)$ for all $t \in T$ and $x > 0$), and if the maximum process $\Max_{i=1}^n X^{(i)}$ of independent copies of $X$ has the same finite dimensional distributions (f.d.d.) as the process $n X$ for each $n \in \NN$. 
The crucial point in the realization problem for TCFs will be the close connection of the TCF $\chi^{(X)}$ of a simple max-stable process $X=\{X_t\}_{t \in T}$ to the \emph{extremal coefficient function (ECF)} $\theta^{(X)}$ of the respective process $X$.  Therefore, let $\finite(T)$ denote the \emph{set of finite subsets} of the space $T$. The ECF $\theta^{(X)}$ of a simple max-stable process $X$ on $T$ is a function on $\finite(T)$ that is given by $\theta^{(X)}(\emptyset):=0$ and
\begin{align*}
  \theta^{(X)}(A):= - \tau \log \PP\Big(\Max_{t \in A} X_t \leq \tau \Big), \qquad \tau > 0,
\end{align*}
in case $A \neq \emptyset$. The \rhs is indeed independent of $\tau>0$ and lies in the interval $[1,|A|]$, where $|A|$ denotes the number of elements in $A$. In fact, the value $\theta^{(X)}(A)$ can be interpreted as the effective number of independent random variables in the collection $\{X_t\}_{t \in A}$ (\cf \cite{smith_90,schlathertawn_02}).  We call the set of all possible ECFs of simple max-stable processes
\begin{align}\label{eqn:Theta}
\Theta(T)=\left\{\,\theta^{(X)}: \finite(T) \rightarrow \RR \pmid 
    X \text{ a simple max-stable process on } T\,
\right\}.
\end{align}
The bounded ECFs will be denoted 
\begin{align}\label{eqn:Theta_bounded}
\Theta_b(T)=\left\{\,\theta \in \Theta(T) \pmid \theta \text{ is bounded}\,\right\}.
\end{align}
In fact, the set of ECFs $\Theta(T)$ can be completely characterized by a property called \emph{complete alternation} (\cf Theorem~\ref{thm:ECF_CA} below). Using  the notation and definition from \cite{molchanov_05}, we set for a function $f:\finite(T) \rightarrow \RR$ and elements $K,L \in \finite(T)$ 
\begin{align*}
\left(\Delta_{K}f\right)(L):= f(L)-f(L\cup K).
\end{align*}
Then a function $f: \finite(T) \rightarrow \RR$ is called \emph{completely alternating} on $\finite(T)$ if for all $n \geq 1$, $\{K_1,\dots, K_n\} \subset \finite(T)$ and $K \in \finite(T)$
\begin{align}\label{eqn:defn_CA} 
\left(\Delta_{K_1}\Delta_{K_2} \dots \Delta_{K_n} f \right)(K) = 
\sum_{I \subset \{1, \dots, n\}} (-1)^{|I|} \,f\left(K \cup \bigcup_{i \in I} K_i\right)
 \leq 0.
\end{align}
This condition can be slightly weakened as in Lemma~\ref{lemma:CA_finite} below. Its proof uses the following auxiliary argument.

\begin{lemma}\label{lemma:subsetinduction}
Let $M$ be a finite set and $f:\finite(M) \rightarrow \RR$ be a function on the subsets of $M$. Let $K,L \subset M$ with $K \cap L =\emptyset$.  Then 
\begin{align}\label{eqn:assertionM}
\sum_{I \subset L} (-1)^{|I|+1} f(K \cup I) = \sum_{J \subset (K \cup L)^c} \Big(\sum_{I \subset L \cup J} (-1)^{|I|+1} f((L \cup J)^c \cup I)\Big).
\end{align}
\end{lemma}

\begin{proof}
Each set $(L\cup J)^c\cup I$ occuring on the r.h.s. can be written as a disjoint union $K \cup A \cup B$, with $A \subset L, B \subset (K\cup L)^c$. Let us consider the terms on the r.h.s.\ with fixed $A \subset L$ and fixed $B \subset (K\cup L)^c$.
If $B = \emptyset$, the only possible $I$ and $J$ leading to such a situation are $I=A$ and $J = (K\cup L)^c$, i.e., one obtains the term on the l.h.s.\ with $I = A$.
If $B \neq \emptyset$, the possibilities can be listed as $I = A \cup (B\setminus C)$ and  $J = (K\cup L\cup C)^c$ for some $C \subset B$. Summing these terms over all $C \subset B$ yields $\sum_{C\subset B}(-1)^{|A|+|B\setminus C|+1} f(K \cup A \cup B)= (-1)^{|A|+1}(1-1)^{|B|}f(K \cup A \cup B) = 0$.
\end{proof}

It follows that for finite sets $M$ (instead of arbitrary $T$) complete alternation can be formulated by bounding the value  $f(M)$ by lower order values $f(L)$ for $L \subset M$ as follows (\cf also \cite{schlathertawn_02}, Ineq.~(12)).

\begin{lemma}\label{lemma:CA_finite}
\begin{enumerate}[a)]
\item
A function $f: \finite(T) \rightarrow \RR$ is {completely alternating} on $\finite(T)$ if and only if for all $\emptyset \neq L \in  \finite(T)$ and $K \in \finite(T)$ with $K \cap L = \emptyset$
\begin{align}\label{eqn:short_CA}
  \sum_{I \subset L} (-1)^{|I|+1} f\left(K \cup I\right) \geq 0.
\end{align}
\item  Let $M$ be a non-empty finite set. Then $f: \finite(M) \rightarrow \RR$ is completely alternating if and only if (\ref{eqn:short_CA}) holds for all  $\emptyset \neq L \subset M$ and $K=L^c$, which is equivalent to 
  \begin{align}
    \label{eqn:finite_CA} \Max_{\substack{L \subset M\\|L| \text{ \normalfont odd}}} \, \sum_{\substack{I \subset L\\ I \neq L}} (-1)^{|I|} f\left(L^c \cup I\right)
    & \leq f(M) 
    \leq
    \Min_{\substack{\emptyset \neq L \subset M \\ |L| \text{ \normalfont even}}} \, \sum_{\substack{I \subset L\\ I \neq L}} (-1)^{|I|+1} f\left(L^c \cup I\right).
\end{align}
\end{enumerate}
\end{lemma}

\begin{proof}
\begin{enumerate}[a)]
\item Note that $\finite(T)$ forms an abelian semigroup \wrt the union operation that is generated already by the singletons $\{t\}$ for $t \in T$ and that $\Delta_{\{t\}}\Delta_{\{t\}}=\Delta_{\{t\}}$. Therefore, it suffices already to require (\ref{eqn:defn_CA}) only for $K_i=\{t_i\}$ for pairwise distinct elements $t_i \in T$ ($i=1,\dots,n$) (\cf \cite{bcr_84}, Proposition 4.6.6). Set $L=\{t_1,\dots,t_n\}$. Hence $f$ is completely alternating on $\finite(T)$ if and only if for all $\emptyset \neq L \in  \finite(T)$ and $K \in \finite(T)$ the inequality (\ref{eqn:short_CA}) holds. Secondly, the expression on the l.h.s.\ of (\ref{eqn:short_CA}) equals automatically $0$ if $K \cap L \neq \emptyset$.
\item
Because of \eqref{eqn:assertionM},
it suffices to check (\ref{eqn:short_CA}) for $\emptyset \neq L \subset M$ and $K = L^c$. Separating $f(M)$ and summarizing the cases where $|L|$ is odd and where $|L|$ is even yields the second equivalence. 
\end{enumerate}
\end{proof}

The following example shows that the concept of complete alternation is closely linked to the distributions of $\{0,1\}$-valued processes.

\begin{example}[\cite{molchanov_05}, p.~52] \label{example:cap_CA}
Let $Y=\{Y_t\}_{t \in T}$ be a stochastic process with values in $\{0,1\}$ and let the function $C^{(Y)}:\finite(T) \rightarrow [0,1]$ be given by $C^{(Y)}(\emptyset) = 0$ and $C^{(Y)}(A) = \PP(\exists \, t \in A \text{ such that } Y_t = 1)$.
Then $C^{(Y)}$ is completely alternating. Conversely, if $C:\finite(T) \rightarrow [0,1]$ is completely alternating with $C(\emptyset) = 0$, then $C$ determines the \fdd of a stochastic process $Y=\{Y_t\}_{t \in T}$ with values in $\{0,1\}$, such that $C^{(Y)}=C$. 
\end{example}

\begin{remark}[\cite{molchanov_05}, p.~10]
From the perspective of the theory of random sets it is more natural to define a functional $C^{\Xi}(K)=\PP(\Xi \cap K \neq \emptyset)$ for a random closed set $\Xi$ on compact sets $K$. In this case, $C^{\Xi}$ will be termed the \emph{capacity functional} of the random closed set $\Xi$ and is not only completely alternating on compact sets, but also upper semi-continuous in the sense that $C^{\Xi}(K_n) \downarrow C^{\Xi}(K)$ for $K_n \downarrow K$. These properties ensure that $\Xi$ can be defined on a sufficiently regular probability space. 
A priori our considerations below do not include any regularity constraints. However, we will come back to {Question~(C)} in Corollary~\ref{cor:chicty} and Remark~\ref{rk:chicty}. 
\end{remark}

\begin{theorem}[\cite{strokorbschlather_15}, Theorem 8] \label{thm:ECF_CA} 
$\phantom{a}$\\
Let $\theta: \finite(T) \rightarrow \RR$ be a function on the finite subsets of $T$. Then
\begin{align*} 
\theta \in \Theta(T) \qquad \Longleftrightarrow \qquad 
  \left\{\begin{array}{l}
    \theta \text{ is completely alternating, }  \\
    \theta(\emptyset)=0,\\  \theta(\{t\})=1 \text{ for } t \in T.
  \end{array}\right.
\end{align*}
If $\theta \in \Theta(T)$, then there exists a simple max-stable process $X=\{X_t\}_{t \in T}$ on $T$ with ECF $\theta^{(X)}=\theta$, whose \fdd are given by 
\begin{align*}
&-\log \, \PP\left(X_{t_i} \leq x_i\,,\, i=1,\dots,m\right) = \\
&\sum_{k=1}^m\,\sum_{1\leq i_1 < \dots < i_k \leq m} \hspace{-3mm} -\Delta_{\{t_{i_1}\}} \dots \Delta_{\{t_{i_k}\}} \theta \, (\{t_1,\dots,t_m\}\setminus \{t_{i_1},\dots,t_{i_k}\}) \Max_{j \in \{i_1,\dots,i_k\}} \hspace{-3mm}{x_j^{-1}}.
\end{align*}
\end{theorem}

If a process $\{X_t\}_{t \in T}$ has the f.d.d.\ stated in Theorem~\ref{thm:ECF_CA}, then it is called \emph{Tawn-Molchanov process (TM process) associated with the ECF $\theta$} henceforth. Note that this convention and the notation from \cite{MolchanovStrokorb} differ in the sense that \cite{MolchanovStrokorb} consider TM processes with at least upper-semi continuous sample paths.
By construction, the class of f.d.d.'s of TM processes on a space $T$ is in a one-to-one cor\-res\-pondence with the set of ECFs $\Theta(T)$.
In fact, if $\theta \in \Theta(T)$ and $X$ is an associated TM process, the process $X$ takes a unique role among simple max-stable processes sharing the same ECF $\theta$ in that it provides a sharp lower bound for the \fdd \citep[Corollary 33]{strokorbschlather_15}.

\begin{corollary}[\cite{strokorbschlather_15}, Corollaries 13 and 14] \label{cor:Theta_convex_compact}
The set of ECFs $\Theta(T)$ is convex and compact \wrt the topology of pointwise convergence on $\RR^{\finite(T)}$. 
\end{corollary}

\noindent The connection of the TCF $\chi^{(X)}$ to the second-order extremal coefficients of a simple max-stable process $X$ is given by
\begin{align}
\notag \chi^{(X)}(s,t) &= 2- \lim_{\tau \to \infty} \frac{1-\PP\left(X_s \leq \tau, X_t \leq \tau \right)}{1-\PP\left(X_t \leq \tau \right)} \\
\label{eqn:TCFfromECF}
&=  2- \frac{\log \PP\left(X_s \leq \tau, X_t \leq \tau \right)}{\log \PP\left(X_t \leq \tau \right)}
=2-\theta^{(X)}(\{s,t\}).
\end{align}
Therefore, it will be convenient to introduce the following map
\begin{align}\label{eqn:psi}
\psi:\RR^{\finite(T)} \rightarrow \RR^{T \times T}, \qquad \psi(F)(s,t):=2-F(\{s,t\}),
\end{align}
such that (\ref{eqn:TCFfromECF}) reads as $\chi^{(X)}=\psi(\theta^{(X)})$. Note that $\psi$ is continuous if we equip both spaces $\RR^{\finite(T)}$ and $\RR^{T \times T}$ with the topology of pointwise convergence.
Finally, we restate a continuity result from \cite{strokorbschlather_15} in terms of TCFs (instead of ECFs as in the reference).

\begin{corollary}[\cite{strokorbschlather_15}, Theorem 25] \label{cor:TMcty} 
$\phantom{a}$\\
Let $X=\{X_t\}_{t \in T}$ be a TM process and $\chi^{(X)}$ its TCF. 
Then the following statements are equivalent:
\begin{enumerate}[(i)]
\item $\chi^{(X)}$ is continuous.
\item $\chi^{(X)}$ is continuous on the diagonal $\{(t,t):t \in T\}$.
\item $X$ is stochastically continuous.
\end{enumerate}
\end{corollary}

\begin{remark}
In fact, a TM process $X=\{X_t\}_{t \in T}$ is always stochastically continuous with respect to the semimetric $\eta^{X}(s,t)=1-\chi^{(X)}(s,t)$.
\end{remark}



\section{TCFs are realized by TM processes}\label{sect:realizedTM}

In order to simplify the realization problem for TCFs (termed as Questions (A) to (E) in the introduction) it is desirable to find a subclass of stochastic processes which can realize any given TCF $\chi$. We denote the set of all TCFs and certain subclasses as follows:
\begin{align*}
\tcf(T)&:=\left\{\chi^{(X)} \pmid \begin{array}{l}\text{$X$ a stochastic process on $T$ with identical} \\ \text{one-dimensional margins and existing \ $\chi^{(X)}$} \end{array}\right\},\\
\tcf_\infty(T)&:=\left\{\chi^{(X)} \in \tcf(T) \pmid \text{$X$ with essential supremum $\infty$}\right\},\\
\maxstab(T)&:=\left\{\chi^{(X)} \in \tcf(T) \pmid \text{$X$ simple max-stable}\right\},\\
\tm(T)&:=\left\{\chi^{(X)} \in \tcf(T) \pmid \text{$X$ a TM process}\right\}.
\end{align*}

\begin{remark}
The class $\tcf_\infty(T)$ represents the TCFs of processes whose margins have no jump at the upper endpoint. To see this, first note that a distribution function $F:\RR\rightarrow [0,1]$ has no jump at its upper endpoint $u \in (-\infty,\infty]$ if and only if there exists a continuous strictly increasing transformation $f:(-\infty,u)\rightarrow \RR$ such that $F \circ f^{-1}$ is a distribution function with upper endpoint $\infty$, and secondly, $\chi^{(X)} = \chi^{(f \circ X)}$ if $X$ is a stochastic process with marginal distribution $F$ and TCF $\chi^{(X)}$.
\end{remark}

\noindent A priori it is clear that
\begin{align}\label{eqn:MAXinTCF}
\tm(T) \subset \maxstab(T)\subset \tcf_\infty(T)\subset \tcf(T).
\end{align}
Further, let us introduce the class of uncentered and normalized covariance functions of binary processes
\begin{align}\label{eqn:BIN_defn}
\bin(T)&:=\left\{(s,t) \mapsto \PP(Y_s = 1 | Y_t = 1) \pmid \begin{array}{l}\text{$Y$ a stochastic process on $T$ with} \\ \text{identical one-dimensional margins}\\ \text{with values in $\{0,1\}$ and $\EE Y_t \neq 0$} \end{array}\right\},
\end{align}
which is closely related to the above classes. By definition of $\tcf(T)$ and considering the processes $Y_t = \Eins_{X_t > \tau}$ indexed by $\tau>0$, we observe
\begin{align}\label{eqn:TCFinSCBIN}
\tcf(T)\,  \subset \text{  sequential closure of  \,} \bin(T),
\end{align} 
where the sequential closure is meant \wrt pointwise convergence. The following theorem gives an affirmative answer to the question whether $\tcf(T)$ and $\maxstab(T)$ coincide \emph{(Question (D) in the Introduction)} and yields also the connection to the other classes. In fact, the class of TM processes can realize already any given TCF.  

\begin{theorem}\label{thm:TCFisMAX}
\begin{enumerate}[a)]
\item For arbitrary sets $T$ the following classes  coincide
\begin{align}
\bin(T)&=\psi(\Theta_b(T)),\\[1mm]
\notag \tcf(T)&=\tcf_\infty(T)=\maxstab(T)=\tm(T)=\psi(\Theta(T))\\
\label{eqn:TCFisMAX} &=\text{  sequential closure of  \,} \bin(T) = \text{  closure of  \,} \bin(T),
\end{align}
where the map $\psi$ is from (\ref{eqn:psi}), $\Theta(T)$ and $\Theta_b(T)$ are from (\ref{eqn:Theta}) and (\ref{eqn:Theta_bounded}), respectively,  and
the (sequential) closure is meant \wrt pointwise convergence.
\item For infinite sets $T$ the inclusion $\bin(T) \subsetneq \tcf(T)$ is  proper.
\item For finite sets $M$ the equality $\bin(M)=\tcf(M)$ holds.
\end{enumerate}
\end{theorem}

\begin{proof}
\begin{enumerate}[a)]
\item
First, we establish $\bin(T)=\psi(\Theta_b(T))$: 

Let $f \in \bin(T)$ and let $Y$ be a corresponding process with values in $\{0,1\}$ as in the definition of $\bin(T)$ (\cf (\ref{eqn:BIN_defn})). 
Let the function $C^{(Y)}:\finite(T) \rightarrow [0,1]$ be given by $C^{(Y)}(\emptyset) = 0$ and $C^{(Y)}(A) = \PP(\exists \, t \in A \text{ such that } Y_t = 1)$
as in Example~\ref{example:cap_CA}. Then $C(\{t\})=\EE Y_t$ lies in the interval $(0,1]$ and is independent of $t \in T$ due to identical one-dimensional margins. 
Further, the function $f$ is given by $f(s,t)=\PP(Y_s = 1 \mid Y_t = 1) = 2-C(\{s,t\})/C(\{t\})$. Now, set $\theta(A):=C(A)/C(\{t\})$ for $A \in \finite(T)$. Then $\theta$ satisfies $\psi(\theta)(s,t)=2-\theta(\{s,t\})=f(s,t)$ and $\theta$ is clearly bounded by $1/C(\{t\})$. It follows from Example~\ref{example:cap_CA} and Theorem~\ref{thm:ECF_CA} that $\theta$ lies in $\Theta(T)$. Hence, $f \in \psi(\Theta_b(T))$.

Conversely, let $\theta \in \Theta_b(T)$ be bounded, say by $\kappa$. Clearly, $\kappa \geq \theta(\{t\})=1$. Set $C(A):=\theta(A)/\kappa$. Then $C$ satisfies all requirements of Example~\ref{example:cap_CA} to determine the f.d.d.\ of  a binary process $Y$ with values in $\{0,1\}$ with $C^{(Y)}=C$. The process $Y$ has identical one-dimensional margins since $\theta(\{t\})=1$ for $t \in T$, and $\EE Y_t = 1/\kappa > 0$. So $Y$ fulfills the requirements of a process in the definition of $\bin(T)$. Finally, note that the corresponding function in $\bin(T)$ is given by $\PP(Y_s = 1 \mid Y_t = 1)=2-C(\{s,t\})/C(\{t\})=\psi(\theta)(s,t)$ as desired.

Secondly, the equality $\maxstab(T)=\tm(T)=\psi(\Theta(T))$ follows directly from Theorem~\ref{thm:ECF_CA}. On the one hand this implies 
\begin{align*}
\bin(T) = \psi(\Theta_b(T)) \subset \psi(\Theta(T))=\tm(T),
\end{align*}
and on the other hand, we obtain that $\tm(T)$ is compact, as it is the image of the compact set $\Theta(T)$ (Corollary~\ref{cor:Theta_convex_compact}) under the continuous map $\psi$. Now, the assertion (\ref{eqn:TCFisMAX}) follows from
\begin{align*}
& \tcf(T) \stackrel{(\ref{eqn:TCFinSCBIN})}{\subset} \text{  sequential closure of  \,} \bin(T) \subset \text{ closure of  \,} \bin(T) \\
& {\subset} \text{ closure of  \,} \tm(T)  \subset \, \tm(T) \stackrel{(\ref{eqn:MAXinTCF})}{\subset}  \, \maxstab(T) 
 \stackrel{(\ref{eqn:MAXinTCF})}{\subset} \, \tcf_\infty(T) 
 \stackrel{(\ref{eqn:MAXinTCF})}{\subset} \, \tcf(T). 
\end{align*}
\item Let $T$ be an infinite set and let $\chi(s,t):=\delta_{st}$. Indeed $\chi$ is an element of $\maxstab(T)$ realized by the simple max-stable process $X$ on $T$, where the variables $\{X_{t}\}_{t \in T}$ are \iid standard Fr{\'e}chet random variables. Suppose that $\chi \in \bin(T)$. Then $\PP(Y_s=1,Y_t=1)=0$ for all $s,t \in T$ with $s \neq t$. Thus, $\PP(\bigcup_{s \in S} \{Y_s=1\})=\sum_{s \in S} \PP(Y_s=1)=\infty$ for any countably infinite subset $S \subset T$, a contradiction.
\item If $M$ is finite,  elements of $\Theta(M)$ are automatically bounded by $|M|$ and thus, $\Theta(M)=\Theta_b(M)$. 
\end{enumerate}
\end{proof}

The latter result does not include any regularity considerations beyond the product topology that is somewhat unnatural in infinite-dimensional stochastic contexts. However, in view of Corollary~\ref{cor:TMcty}, it is possible to identify the role of continuous TCFs in this realization problem and hence address {Question (C)} as follows.

\begin{corollary}\label{cor:chicty}
Let $\chi \in \tcf(T)$. Then the following statements are equivalent.
\begin{enumerate}[(i)]
\item $\chi$ is continuous.
\item $\chi$ is continuous on the diagonal $\{(t,t) \,:\, t \in T\}$.
\item There exists a stochastically continuous stochastic process $\{X_t\}_{t \in T}$\\ with TCF $\chi^{(X)}=\chi$. 
\end{enumerate}
\end{corollary}

\begin{remark}\label{rk:chicty}
In fact, any TM process $X$ with continuous TCF $\chi^{(X)}$ is stochastically continuous. It follows from \citeauthor{dehaan_84}'s  (\citeyear{dehaan_84}) construction that any simple max-stable process on $\RR^d$ (or any other locally compact second countable Hausdorff space) that is continuous in probability, can be realized on a sufficiently regular probability space. Hence, this applies to TM processes with continuous TCFs, since they are simple max-stable and continuous in probability by the preceding corollary. 
\end{remark}

\begin{remark}
\cite{lm15} discuss regularity conditions on the two-point covering function of a random set, or equivalently, a unit covariance function (cf.\ Section \ref{sect:corcut}) that ensure the existence of a realizing closed set, or equivalently, a realizing $\{0,1\}$-valued process with upper semi-continuous paths. Here, we do not know which regularity conditions on the TCF ensure the existence of a realizing upper semi-continuous process.
\end{remark}

\section{Basic closure properties and characterization by inequalities}\label{sect:props}

Finally, we collect some immediate and important consequences concerning operations on the set of TCFs and the characterization of the set of TCFs by means of finite-dimensional projections.

Even though not all non-negative correlation functions are TCFs, both classes have some desirable properties in common as we shall see next.
Well-known operations on (non-negative) correlation functions include convex combinations, products and pointwise limits. 
Interestingly, the same operations are still admissible for TCFs (answering Question E).

\begin{corollary}\label{cor:basicoperations}
The set of tail correlation functions $\tcf(T)$ is convex, closed under pointwise multiplication and compact \wrt pointwise convergence.
\end{corollary}

\begin{proof} These closure properties follow from Theorem~\ref{thm:TCFisMAX}. Convexity and compactness of $\tcf(T)=\psi(\Theta(T))$ are immediate taking additionally Corollary~\ref{cor:Theta_convex_compact} into account. Moreover, let $\chi_1$ and $\chi_2$ be in $\tcf(T)=\tcf_{\infty}(T)$ with corresponding processes $X^{(1)}$ and $X^{(2)}$ with upper endpoint $\tau_{\text{up}}=\infty$. We choose them to be independent and set $X^{(3)}:=X^{(1)} \wedge X^{(2)}$, which then also has upper endpoint $\infty$ and satisfies
\begin{align*}
\PP(X^{(3)}_s \geq x \,|\, X^{(3)}_t \geq x) = \PP(X^{(1)}_s \geq x \,|\, X^{(1)}_t \geq x) \cdot \PP(X^{(2)}_s \geq x \,|\, X^{(2)}_t \geq x).
\end{align*}
Consequently, the TCF $\chi_3$ of $X^{(3)}$ is the product $\chi_3=\chi_1 \cdot \chi_2$. 
\end{proof}



\noindent Secondly, the set of TCFs can be characterized through finite-dimensional projections.

\begin{corollary}\label{cor:TCFextension}
A real-valued function $\chi: T\times T \rightarrow \RR$ belongs to $\tcf(T)$ if and only if the restriction $\left.\chi\right|_{M \times M}$ belongs to $\tcf(M)$ for all non-empty finite subsets $M$ of $T$.
\end{corollary}

\begin{proof}
If $\chi \in \tcf(T)$, then necessarily $\chi|_{S\times S}\in \tcf(S)$ for any subset $S \in T$. To show the reverse implication, let $\chi|_{M\times M} \in \tcf(M)$ for all $M \in \mathcal{F}(T)\setminus \{\emptyset\}$. Since $\tcf(T) \subset [0,1]^{T\times T}$ is closed, to prove $\chi \in \tcf(T)$ it suffices to show that $U \cap \tcf(T) \neq \emptyset$ for any open neighborhood $U$ of $\chi$ in $[0,1]^{T\times T}$. Given $U$, there is a finite subset of $T\times T$, which we may assume to be of the form $M \times M$, and open sets $A_{(i,j)}\subset [0,1]$, $(i,j) \in M\times M$, such that $\chi \in \bigcap_{(i,j) \in M \times M} \pr_{(i,j)}^{-1}\left(A_{(i,j)}\right) \subset U$ (where $\pr_{(s,t)}:[0,1]^{T \times T} \rightarrow [0,1]$ denotes the natural projection). Since $\chi|_{M\times M}$ trivially extends to an element $\tilde\chi \in \tcf(T)$ (e.g.\ copy one of the random variables), we have $\tilde\chi \in U  \cap \tcf(T) \neq \emptyset$. 
\end{proof}

In Part II of this exposition, we will see that for a finite set $M$, the set of TCFs $\tcf(M)$ constitutes a convex polytope in $\RR^{\lvert M \rvert \times \lvert M \rvert}$ that can be described by means of a finite system of (affine) inequalities.
In this regard Corollary~\ref{cor:TCFextension} shows that for an arbitrary set $T$,  the class $\tcf(T)$ may also be completely characterized by a system of (affine) inequalities. This is not evident since elements of $\tcf(T)$ are defined a priori through a limiting procedure.

\vspace{7mm}

{\noindent \textbf{\large Part II\\[2mm]
The realization problem for TCFs on finite sets}}\\[3mm]

\noindent In view of Corollary~\ref{cor:TCFextension} it suffices to study $\tcf(M)$ for finite sets $M$ if one is interested in a complete characterization of the space $\tcf(T)$ for arbitrary $T$. Therefore, we focus on a non-empty finite set $M=\{1,\dots,n\}$ in this section and set 
 \begin{align*}
\tcf_n:=\tcf(\{1,\dots,n\}).
\end{align*} 
To begin with, we show that $\tcf_n$ can be viewed as a convex polytope in Section~\ref{sect:TCFpolytope}. Its geometry will be studied subsequently. Here, we start off with some basic observations and low-dimensional results in Section~\ref{sect:basicfacts}. 
Section~\ref{sect:generalresults} collects more sophisticated results on $\tcf_n$ with deeper insights into the rapidly growing complexity of $\tcf_n$ as $n$ grows, including connections between $\tcf_n$ and $\tcf_{n'}$ for $n' > n$.
Thereby, some 
obervations from Section~\ref{sect:basicfacts} will be uncovered as low-dimensional phenomena.
At least, it is possible to identify the precise relation of $\tcf_n$ to the so-called cut- and correlation-polytopes as well as to the polytope of unit covariances. 
To complement these general observations, Section~\ref{sect:compresults} reports all results relying on software computations and, in particular, all combinatorial considerations that were necessary in order to push the entire description of the vertices and facets of $\tcf_n$ up to $n \leq 6$. 
Finally, we pursue some open questions on the geometry on $\tcf_n$ in Section~\ref{sect:openquestions}.

\section{TCF$_n$ is a convex polytope}\label{sect:TCFpolytope}

Elements of $\tcf_n$ are functions on $\{1,\dots,n\} \times \{1,\dots,n\}$, that is to say, they are $n\times n$ matrices. Since TCFs are symmetric and take the value $1$ on the diagonal, we may regard $\tcf_n$ for $n \geq 2$ as a subset of 
\begin{align*}
\RR^{E_n} \cong \RR^{n \choose 2}=\RR^{n(n-1)/2},
\end{align*} 
where $E_n$ is the set of edges of the \emph{complete graph} $K_n$ with vertices $V_n=\{1,\dots,n\}$. It will be convenient to interpret elements of $\tcf_n$ as an edge labelling of $K_n$, which is why we call $K_n$ the \emph{support graph} for $\tcf_n$. Due to Theorem~\ref{thm:TCFisMAX} and (\ref{eqn:BIN_defn}) we know already 
\begin{align}\label{eqn:TCFisBIN}
\tcf_n=\bin_n:=\left\{\chi \in \RR^{E_n} \pmid  
\begin{array}{l}
\chi_{ij} = {\EE(Y_i Y_j)}/{ \EE Y_j} \text{ where } \\
\text{$Y_1,\dots,Y_n$ take values in $\{0,1\}$} \\ 
\text{and $\EE Y_1 = \ldots = \EE Y_n > 0$} 
\end{array}
\right\}.
\end{align}
The following lemma is a reformulation of this fact and will be useful later on.

\begin{lemma}\label{lemma:TCFevents}
An element $\chi \in \RR^{E_n}$ belongs to  $\tcf_n$ if and only if it can be written as 
\begin{align*}
\chi_{ij}=\PP(A_i|A_j), \quad 1 \leq i<j \leq n
\end{align*} 
for some (finite) probability space $(\Omega,\A,\PP)$ and measurable subsets $A_1,\dots,A_n \in \A$ which satisfy $\PP(A_1)=\dots=\PP(A_n)>0$.
\end{lemma}

\begin{remark}\label{remark:TCFevents}
In Lemma~\ref{lemma:TCFevents} we may assume that $\PP(A_1)=\dots=\PP(A_n)=c$ for any constant $0 < c \leq 1/n$: Otherwise enlarge $\Omega$, such that $A :=\bigcup_{i=1}^n A_{i} \neq \Omega$. On $A$ define the measure $\QQ :=c/\PP(A_{1})\cdot \PP|_{A}$. Then $\QQ(A) \leq 1$ and, thus, $\QQ$ extends to a probability measure on $\Omega$ with $\QQ(A_{i}) = c$ and $\QQ(A_{i} | A_{j})=\chi_{ij}$. 
\end{remark}

\noindent Likewise, we set 
\begin{align*}
\Theta_n:=\Theta(\{1,\dots,n\}) 
\end{align*}
and, since $\theta_\emptyset = 0$ and $\theta_i = 1$ for $i=1,\dots,n$, we may regard $\Theta_n$ for $n \geq 2$ as a subset of 
\begin{align*}
\RR^{\finite^{(2)}_n} \cong \RR^{2^n-n-1},
\end{align*} 
where $\finite^{(2)}_n$ is the set of subsets of $V_n$ with at least two elements. Remember from (\ref{eqn:psi}) that
\begin{align}\label{eqn:psin}
\tcf_n = \psi_n(\Theta_n) \qquad \text{where} \qquad \psi_n: \RR^{\finite^{(2)}_n}  \rightarrow \RR^{E_n}, \quad \psi_n(\theta)_{ij} = 2 - \theta_{ij},
\end{align}
and note that  $\psi_n=2-\pr_{E_n}$ is essentially a projection onto the $n \choose 2$ coordinates of $\RR^{E_n}$. Before we proceed, we need to revise some notation for convex polytopes.

\paragraph{\textbf{\upshape Notation and facts concerning convex polytopes}} 
(\cf \cite{ziegler_95}).\\
A subset $P \subset \RR^p$ is a \emph{convex polytope} if $P$ is bounded and can be represented as $P=\{x \in \RR^p \,:\, Cx \leq c\}$ for a $q \times p$ matrix $C$ and a $q$-vector $c$ for some $q \in \NN$ (where $\leq$ is meant componentwise). The rows of $C$ and $c$ represent hyperplanes in $\RR^d$ and the inequality $\leq$ determines the corresponding halfspace to which $P$ belongs. The system $Cx \leq c$ will be called an \emph{$\mathcal{H}$-representation} (or \emph{halfspace representation}) of $P$.

An $\mathcal{H}$-representation will be called a  \emph{facet representation} if it is minimal in the sense that  none of the rows in $C$ and $c$ can be deleted in order to define $P$, \ie $P \neq \{x \in \RR^p \,:\, C_{-i} x \leq c_{-i}\}$ for all $i=1,\dots,q$, where $C_{-i}$ and $c_{-i}$ are the modified versions of $C$ and $c$ with the $i$-th row removed.  In fact, an $\mathcal{H}$-representation $Cx\leq c$ is a {facet representation} if every row of $C$ and $c$ yields in fact a \emph{facet inducing} inequality of $P$, where an inequality  $C_{i} x \leq c_{i}$ is facet inducing if $\dim(P \cap \{x \in \RR^p \,:\, C_{i} x = c_{i}\}) = \dim(P) - 1$. The latter is equivalent to the existence of  $\dim(P)$ affinely independent points $x^1,\dots,x^{\dim(P)} \in P$ solving the equation $C_{i} x = c_{i}$.
By a slight abuse of notation, we will usually refer to the inequality $C_{i} x \leq c_{i}$ as a \emph{facet} of $P$ if it induces a facet (instead of calling the set $P \cap \{x \in \RR^p \,:\, C_{i} x = c_{i}\}$ a facet).

Equivalently, a subset $P \subset \RR^p$ is a convex polytope if $P$ equals the convex hull of a finite subset $S \subset \RR^p$. Then $S$ will be called a \emph{$\mathcal{V}$-representation} of $P$. A minimal $\mathcal{V}$-representation, with respect to set inclusion, will be called a \emph{vertex representation}. In fact, the vertex representation is unique and given by the set  $\Ex(P)$ of extremal points, or \emph{vertices}, of $P$, \ie the points of $P$ that cannot be decomposed non-trivially as a convex combination of two other points of $P$.  Note that in general a  $\mathcal{V}$-representation of $P$ may consist of more points than the vertex set $\Ex(P)$.

Moreover, if $P \subset \RR^p$ is a convex polytope and $\pi:\RR^p \rightarrow \RR^{p'}$ is an affine map $x \mapsto Ax+b$, then the image $\pi(P)$ is again a convex polytope and secondly, any intersection of $P$ with an affine subspace of $\RR^p$ is a convex polytope.

\begin{corollary}\label{cor:finitepolytope}
For all $n \in \NN$ the sets $\Theta_n$ and $\tcf_n$ are convex polytopes.
\end{corollary}

\begin{proof}
For $\Theta_n$ this property is evident from Theorem~\ref{thm:ECF_CA} and (\ref{eqn:finite_CA}). 
But then the affine map $\psi_n$ maps $\Theta_n$ to the convex polytope $\tcf_n=\psi_n(\Theta_n)$.
\end{proof}

\noindent Now, that we know that $\tcf_n$ is a convex polytope, we seek to understand its geometric structure. At best, we would like to determine its vertex and facet representation (and we will indeed do so in Section~\ref{sect:compresults} up to $n \leq 6$). 
To repeat the terminology adopted from convex geometry in this context, note that an $\mathcal{H}$-representation of $\tcf_n$ (and in particular, a facet representation) allows one to check whether a given matrix is indeed a TCF, since any $\mathcal{H}$-representation of $\tcf_n$ constitutes a  set of necessary and sufficient conditions for being a TCF. In a facet representation no condition is obsolete.  
Complementary, a $\mathcal{V}$-representation (and in particular, a vertex representation) of $\tcf_n$ is more useful if one wants to generate valid TCFs. Any TCF can be obtained as a convex combination of the elements of a $\mathcal{V}$-representation. 
In a vertex representation no point is obsolete.  

\section{Basic observations and low-dimensional results for TCF$_{n}$}\label{sect:basicfacts}

This section comprises two first general observations. First, every polytope $\tcf_n$ satisfies a certain system of inequalities (to be called \emph{hypermetric inequalities}) and, second, we identify its $\{0,1\}$-valued vertices as  so-called \emph{clique partition points}. With regard to the explicit vertex and facet structure of $\tcf_n$ in low dimensions, both findings might lead to tempting conjectures on the geometry of $\tcf_n$  eventually refuted by the more sophisticated methods applied in Section~\ref{sect:generalresults}.


\paragraph{\textbf{\upshape Hypermetric inequalities}} 
Remember that we identified the set of all TCFs on $V_n=\{1,\dots,n\}$ with a subset of $\RR^{E_n} = \RR^{\binom{n}{2}}$ while it originally was interpreted as a set of symmetric $n \times n$ matrices with 1's on the diagonal. In the sequel we will identify points $x=(x_{ij})_{1\leq i < j\leq n} \in \RR^{E_n}$  with $n \times n$ matrices $(x_{ij})_{1\leq i,j \leq n}$ via $x_{ji} =x_{ij}$ and $x_{ii} := 1$. 

Let $b =(b_{1},\ldots,b_{n}) \in \ZZ^n$. The point $(x_{ij})_{1\leq i < j\leq n}\in \RR^{E_n}$ satisfies the \emph{hypermetric inequality defined by} $b$ if
\begin{align}
\notag \sum_{1\leq i,j\leq n}b_{i}b_{j}x_{ij} &\geq \sum_{i=1}^{n}b_{i}\\
\text{or, equivalently,} \qquad
\label{eq:hypermetric} \sum_{1\leq i < j \leq n} (-b_{i}b_{j}) x_{ij} &\leq \frac{1}{2}\sum_{i=1}^n b_{i}(b_{i} -1).
\end{align}

\begin{remark}
In \cite{dezalaurent_97} 
the inequalities $\sum_{1\leq i<j\leq n}b_{i}b_{j}x_{ij} \leq 0$ with $\sum_{1\leq i \leq n}b_{i} = 1$ are termed hypermetric. All these inequalities are valid for the cut polytope $\cut_n$ to be introduced here in Section~\ref{sect:corcut} \cite[Lemma 28.1.3]{dezalaurent_97}. For $\tcf_n$ the variant \eqref{eq:hypermetric} is an appropriate ``counterpart''.
\end{remark}

\begin{lemma}\label{lemma:hypervalid} All hypermetric inequalities (in the sense of \eqref{eq:hypermetric}) are valid for elements of $\tcf_{n}$.
\end{lemma}

\begin{proof}
Let $Y_{1},\ldots,Y_{n}$ be a $\{0,1\}$-valued stochastic model for $\chi \in \tcf_{n}$. Set $a:=\EE(Y_1)>0$. Then for $b \in \ZZ^n$ 
\begin{align*}
\sum_{1\leq i,j\leq n}\hspace{-1mm} b_{i}b_{j}\chi_{ij} = 
\hspace{-1mm} \sum\limits_{1 \leq i,j\leq n} \hspace{-1mm} b_{i}b_{j}\frac{\EE(Y_{i}Y_{j})}{a} = \frac{1}{a} \EE\bigg[\sum\limits_{i=1}^n b_{i}Y_{i}\bigg]^2 \geq \frac{1}{a} \EE\bigg[\sum\limits_{i=1}^n b_{i}Y_{i}\bigg] =  \sum\limits_{i=1}^n b_{i}, \qquad 
\end{align*}
as for any integer $k$ we have $k^2 \geq k$. 
\end{proof}

\paragraph{\textbf{\upshape Clique partition polytopes}} 

A subset $\{C_1,\dots,C_k\}$ of the powerset of $V_n=\{1,\dots,n\}$ is a \emph{partition} of $V_n$ if $k \geq 1$, $C_r\cap C_s = \emptyset$ for $r \neq s$ and $\bigcup_{r=1}^k C_r = V_n$. A partition of $V_n$ defines a \emph{clique partition point} $\gamma(\{C_1,\dots,C_k\}) \in \{0,1\}^{E_n}$ by 
\begin{align*}
\gamma(\{C_1,\dots,C_k\})_{ij} = \sum_{r=1}^{k} \Eins_{\{i,j\} \subset C_r}, \quad 1 \leq i < j \leq n.
\end{align*} 
The \emph{clique partition polytope} is defined as the convex hull of the clique partition points \citep{groetschelwakabayashi_90} in $\RR^{E_n}$
\begin{align*}
\cpp_n := \conv\left(\{ \gamma(\{C_1,\dots,C_k\}) \,:\, \{C_1,\dots,C_k\}  \text{ partition of $V_n$} \}\right).
\end{align*}
Being $\{0,1\}$-valued, the clique partition points are automatically the extremal points of their convex hull:
\begin{align*}
\Ex\left(\cpp_n\right)=\left(\{ \gamma(\{C_1,\dots,C_k\}) \,:\, \{C_1,\dots,C_k\}  \text{ partition of $V_n$} \}\right).
\end{align*}
It turns out that all $\{0,1\}$-valued vertices of $\tcf_n$ are precisely the clique partition points. 

\begin{proposition}\label{prop:cpp} 
$\tcf_n \cap \{0,1\}^{E_n} = \Ex(\cpp_n)$ for all $n \in \NN$. In particular $\cpp_n \subset \tcf_n$.
\end{proposition}

\begin{proof}
 Since $\tcf_{n} \cap \{0,1\}^{E_n} \subset \Ex(\tcf_{n})$ it suffices to show the first statement. For $n=2$ we have $\tcf_{2} = [0,1]$ and $\{0,1\} = \Ex(\cpp_{2})$. For $n\geq3$ the points in $\tcf_{n}$ have to satisfy the triangle-inequalities (all permutations of 
$\chi_{1,2}+\chi_{2,3}-\chi_{1,3} \leq 1$,
see (\ref{eqn:triangle}) and also \eqref{eq:hypermetric} with $b=(1,-1,1,0,\ldots,0)$). For points $\chi \in \tcf_{n} \cap \{0,1\}^{E_n}$, viewed via the support graph $K_n$, this implies for any triple of nodes $i,j,k$, where the edges $\{i,j\}$ and $\{j,k\}$ have value 1, that also the edge $\{i,k\}$ has value 1. Thus, a simple inductive argument shows: for any pair of nodes $i,j$, which are connected by a path of edges with value 1, the edge from $i$ to $j$ has also value 1. This shows that the points in $\tcf_{n} \cap \{0,1\}^{E_n}$ are clique partition points. 
In order to see that any clique partition point $\gamma(\{C_1,\dots,C_k\})$ belongs to $\tcf_{n} \cap \{0,1\}^{E_n}$ choose $\Omega = \{1,\ldots,k\}$ with  uniform distribution $\PP$ and $A_{i} = \{r_{i}\}$, $1 \leq i \leq n$, with $r_{i}$ uniquely determined by $i \in C_{r_{i}}$ and apply Lemma~\ref{lemma:TCFevents}.
\end{proof}
For $n \leq 4$ the clique partition polytope and $\tcf_n$ even coincide. 

\begin{proposition}\label{prop:cpplowdim} 
$\tcf_n = \cpp_n$ for $n \leq 4$.
\end{proposition}

\begin{proof}
For $n\leq 4$ we computed explicitly that $\Ex(\tcf_n)=\Ex(\cpp_n)$ from the characterization (\ref{eqn:psin}) \citep[Tables~3.1 and  3.3]{strokorb_13} and confirmed this result using the software \texttt{polymake}. This implies $\tcf_n = \cpp_{n}$ for $n \leq 4$.
\end{proof}

\paragraph{\textbf{\upshape Low-dimensional phenomena}}
Even though for $n \leq 4$ the polytope $\tcf_n$ and the clique partition polytope $\cpp_n$ coincide, the property $\tcf_n = \cpp_n$ will turn out to be a low-dimensional phenomenon. Starting from $n=5$ the vertices of $\tcf_n$ are not $\{0,1\}$-valued anymore (see Corollary \ref{cor:nocpp} in Section~\ref{sect:generalresults}), in particular $\cpp_n \subsetneq \tcf_n$ for $n \geq 5$. Still, up to $n \leq 5$ all facet inducing inequalities of $\tcf_n$ turn out to be hypermetric and one might be tempted to believe that certain hypermetric inequalities provide an $\mathcal{H}$-representation for $\tcf_n$ also in higher dimensions. Again, this property constitutes only another low-dimensional phenomenon. Starting from  $n=6$ not all facets of $\tcf_n$ are hypermetric anymore (see Proposition \ref{prop:nonhypfacets} in Section~\ref{sect:generalresults}).


\section{Sophisticated results on the geometry of $\tcf_n$}\label{sect:generalresults}

A fundamental observations in this section concerns the lifting of vertices and facets to higher dimensions (Section~\ref{sect:lifting}). It means that vertices (and facets) of $\tcf_n$ will also appear as vertices (and facets) of $\tcf_{n'}$ for $n'>n$ if the coordinates (or coefficients) are filled up with zeros at appropriate places. Note that both statemenents are not evident, but a deep structural result only revealed by some delicate combinatorial arguments. Subsequently, we prove that every rational number in the interval $[0,1]$ will appear as coordinate value in the vertex set of $\tcf_n$ starting from a sufficiently large $n$ (Proposition~\ref{prop:unbounded} in Section~\ref{sect:unbounded}) and that $\tcf_n$ possesses non-hypermetric facets starting from $n \geq 6$ (Proposition~\ref{prop:nonhypfacets} in Section~\ref{sect:nonhypfacets}).
Taken together, these results give insights into the rapidly growing complexity of $\tcf_n$ as $n$ grows and confound the aim of a full description of vertices and facets of $\tcf_n$ for arbitrary $n$.
Finally, Section~\ref{sect:corcut} provides an alternative (``dual'') description of the polytope $\tcf_n$ (which we recognized already as the \emph{projection} of the polytope $\Theta_n$) as an \emph{intersection} with the so-called correlation polytope or, equivalently, with the so-called cut-polytope.


\subsection{\textbf{\upshape Lifting of vertices and facets to higher dimensions}}\label{sect:lifting}

First, we deal with connections between  $\tcf_{n}$ and $\tcf_{n+1}$.
A particularly important feature is the lifting property. That is every vertex of $\tcf_n$ will appear again in the list of vertices of $\tcf_{n+1}$ with some zeros added.

\begin{lemma}[Projections and liftings of points and vertices]\label{lemma:zeroliftingvertices}
$\phantom{a}$\\
For $\chi \in \tcf_{n+1}$ let $\chi|_{K_{n}}$ denote the restriction of $\chi$ to the subgraph $K_{n} \subset K_{n+1}$ (delete all $\chi_{i,n+1}$, $1 \leq i \leq n$). Conversely, let $\chi^0 \in \RR^{E_{n+1}}$ denote the extension of a point $\chi \in \tcf_{n}$ by 
\begin{align*}
\chi^0_{i,n+1} = 0, \quad 1 \leq i \leq n.
\end{align*} 
\begin{enumerate}[a)]
\item  The assignment $\chi \mapsto \chi|_{K_n}$ maps $\tcf_{n+1}$ onto $\tcf_{n}$.
\item The assignment $\chi \mapsto \chi^0$ embeds $\tcf_{n}$ into $\tcf_{n+1}$ and $\Ex(\tcf_{n})$ into $\Ex(\tcf_{n+1})$.
\item If $ \chi \in \Ex(\tcf_{n+1})$ and $\chi_{i,n+1}=0$ for all $1 \leq i \leq n$, then  $\chi|_{K_{n}} \in \Ex(\tcf_{n})$.
\end{enumerate}
\end{lemma}
\begin{proof}
\begin{enumerate}[a)]
\item Let $Y_{1},\ldots, Y_{n+1}$ be a binary process that models $\chi$. Simply deleting $Y_{n+1}$ gives a model for $\chi|_{K_{n}} \in \tcf_{n}$. Surjectivity follows from b).
\item Let $Y_{1},\ldots, Y_{n}$ be a binary process that models $\chi$. Let $a = \EE(Y_{1})$. Add a disjoint point $\omega_{0}$ to the underlying probability space $\Omega$ and replace the probability measure $\PP$ by  $\frac{1}{1+a}\cdot \PP|_{\Omega} + \frac{a}{1+a}\cdot \delta_{\omega_{0}}$. Extend $Y_{1},\ldots,Y_{n}$ by 0 on $\omega_{0}$, let $Y_{n+1} = \eins_{\{\omega_{0}\}}$. Now, $Y_{1},\ldots,Y_{n+1}$ is a model for $\chi^0$, since $Y_{i}Y_{n+1} = 0$, $1 \leq i \leq n$.
If $\chi^0 \notin \Ex(\tcf_{n+1})$, there is a representation $\chi^0 = \lambda y + (1-\lambda)z$, with $y,z\in \tcf_{n+1}$, $0 < \lambda<1, y \neq z$. 
Since $\chi^0$ is zero on the new edges, the points $y,z$ also have to be zero on the new edges, so $y|_{K_{n}} \neq z|_{K_{n}}$ and $y|_{K_{n}},z|_{K_{n}} \in \tcf_{n}$ by a). Thus, $\chi = \chi^0|_{K_{n}} \notin \Ex(\tcf_{n})$.
\item If $\chi|_{K_{n}} \notin \Ex(\tcf_{n})$, then  $\chi|_{K_{n}} = \lambda y + (1-\lambda)z$, with $y,z\in \tcf_{n}$, $0 < \lambda<1, y \neq z$. By b) we know $y^0, z^0 \in \tcf_{n+1}$. 
Since  $\chi_{i,n+1}=0$ for all $1 \leq i \leq n$, we have $\chi = (\chi|_{K_{n}})^{0} = \lambda y^0 + (1-\lambda)z^0 \notin \Ex(\tcf_{n+1})$.
\end{enumerate}
\end{proof}

\noindent We call $\chi^0$ a \emph{lifting} of $\chi$. The following lemma generalizes the lifting of vertices and will be applied to deduce Proposition~\ref{prop:unbounded}.

\begin{lemma}[Lifting of vertices arising from partitions] \label{lemma:unionsofvertices}
$\phantom{a}$\\
Let $C_{1},\ldots,C_{k}\subset  V_n$ be disjoint subsets of the vertex set $V_n=\{1,\ldots,n\}$ each containing at least two elements of $V_n$.
For $1\leq r \leq k$  let $\chi^{r} \in \Ex(\tcf(C_r))$. Similarly to the interpretation of TCFs on $V_n=\{1,\dots,n\}$ as elements of $\RR^{E_n}$, we interpret $\chi^r$ as an element of $\RR^{E(C_r)}$, where $E(C_r)$ is the set of edges of the complete graph with vertex set $C_r \subset V_n$.
Define $\chi \in \RR^{E_{n}}$ by 
\begin{align*}
\chi_{ij} 
= \left\{ 
\begin{array}{ll}
\chi^r_{ij} & \quad \text{\normalfont if } \{i,j\}  \subset C_r  \text{ \normalfont for some } 1 \leq r \leq k, \\
0 & \quad \text{\normalfont else}.
\end{array}
\right.
\end{align*}
Then $\chi \in \Ex(\tcf_{n})$.
\end{lemma}

\begin{proof}
Because of the lifting property (Lemma~\ref{lemma:zeroliftingvertices}), it suffices to consider the case $V_n=\bigcup_{r=1}^k C_r$, where $C_r=\{i^{(r)}_1,\dots,i^{(r)}_{|C_r|}\}$.
First, we show that $\chi \in \tcf_{n}$. 
To this end, choose (finite) set models 
\begin{align*}
(\Omega_{r}, \PP_{r}), \quad A^r_{i^{(r)}_1},\ldots,A^r_{i^{(r)}_{|C_{r}|}} \subset \Omega_{r} , \quad 1 \leq r \leq k
\end{align*} 
for $\chi^{r}$  as in Lemma~\ref{lemma:TCFevents} such that $\chi^{(r)}_{ij} = \PP(A^r_{i} \,|\, A^r_{j})$. 
By Remark~\ref{remark:TCFevents} these models can be chosen such that $\PP_{r}(A^r_{i})$ does not depend on $r$. Then a stochastic model for $\chi$ is obtained through the normalized disjoint union of these models, \ie where $\Omega=\bigcup_{r=1}^k \Omega_r$, 
$\PP=\frac{1}{k}\sum_{r=1}^k \PP_r(\cdot \cap \Omega_r)$ and $A_i=A^r_i$ if $i \in C_r$. (Note that for each $i\in V_n$ there exists a unique $r$ with $i \in C_r$, since the sets $C_r$ are disjoint and cover $V_n$.)

Now, we show that $\chi \in \Ex(\tcf_{n})$. Suppose not. Then $\chi = \lambda y + (1-\lambda)z$ with $1 < \lambda <0$ and $ y,z \in \tcf_{n}$ with $y \neq z$. Necessarily $y_{ij} = 0$ and $z_{ij}= 0$ whenever $\chi_{ij}=0$. Thus, $y|_{K_n^{r}} \neq z|_{K_n^{r}}$ for some $1 \leq r \leq k$ when $K^r_n$ denotes the complete subgraph of $K_n$ defined by $C_r$. 
Since $y|_{K_n^{r}}, z|_{K_n^{r}} \in \tcf(C_r)$ by Lemma~\ref{lemma:zeroliftingvertices}, we obtain $\chi^r=\chi|_{K_n^{r}} = \lambda y|_{K_n^{r}} + (1-\lambda)z|_{K_n^{r}}$ contradicting $\chi^{r} \in \Ex(\tcf(C_r))$.
\end{proof}

In order to deduce the lifting property also for inequalities and facets, we adapt ideas from \cite[Lemma~26.5.2]{dezalaurent_97}. 
We show that, starting from $n = 3$, no facet inducing inequality will ever become obsolete as $n$ grows. For instance, the triangle inequality $(\ref{eqn:triangle})$ cannot be deduced from a set of other valid inequalities for $\tcf_n$.
One needs $n \geq 3$, since the inequality $\chi_{12} \leq 1$, although facet-inducing for $n=2$, is no longer facet-inducing for $n\geq3$, see Table~\ref{table:hypermetricTCFfacets} and Proposition~\ref{prop:zeroliftingfacets} b).


\begin{proposition}[Lifting of valid inequalities and facets]\label{prop:zeroliftingfacets}
$\phantom{a}$\\
Suppose that 
\begin{align}\label{eqn:tcfnfacet}
a_{0} + a_{1,2}\chi_{1,2}+ \ldots +a_{n-1,n}\chi_{n-1,n} \geq 0
\end{align} 
is a valid inequality for $\tcf_{n}$. The lifting of this inequality to $\RR^{E_{n+1}}$ is the corresponding inequality which is extended by 
\begin{align*}
a_{i,n+1}= 0, \quad 1 \leq i \leq n.
\end{align*} 
\begin{enumerate}[a)]
\item Every lifting of a valid inequality of $\tcf_{n}$ defines a valid inequality of $\tcf_{n+1}$. 
\item For $n \geq 3$, the lifting of a facet of $\tcf_{n}$ defines a facet of $\tcf_{n+1}$.
\end{enumerate}
\end{proposition}

\begin{proof} 
\begin{enumerate}[a)]
\item The lifting of a valid inequality for $\tcf_n$ is always valid for $\tcf_{n+1}$, even for $n=2$, since  the lifted equation applied to $\chi \in \tcf_{n+1}$ returns the same value as the orginal equation applied to $\chi|_{K_{n}}$, which is a point of $\tcf_{n}$, see Lemma~\ref{lemma:zeroliftingvertices}.
\item Now suppose that (\ref{eqn:tcfnfacet}) is a facet for $\tcf_{n}$. By the above, its lifting is a valid inequality for $\tcf_{n+1}$. We show that it defines a facet if $n \geq 3$. First, note that there has to be a coefficient $a_{i,j} \neq 0$. Since $n\geq 3$, there is some index $k \notin \{i,j\}$. To simplify notation, we assume $k=1 < i < j \leq n$. 

Further, let $m := {\binom{n}{2}}$ and let $a = (a_{0},a_{1,2},a_{1,3},\ldots,a_{n-1,n}) \in \RR^{m+1}$ denote the vector of coefficients that appear in the inequality (\ref{eqn:tcfnfacet}). Since (\ref{eqn:tcfnfacet}) induces a facet of $\tcf_n$, there exist $m$ affinely independent points $\chi^k \in \tcf_{n} \subset \RR^{m}$, $ 1 \leq k \leq m$ that solve the inequality (\ref{eqn:tcfnfacet}) as an equation. Affine independence of the $m$ points $\chi^k$ means that the $m$ points $(1,\chi^k) \in \RR^{m+1}$ are linearly independent in $\RR^{m+1}$. By assumption, they solve $\langle (1,\chi^k),a \rangle = 0$, $1 \leq k \leq m$. Let $W \subset \RR^{m+1}$ denote the vector space spanned by  $(1,\chi^k), 1 \leq k \leq m$. Then $\dim(W) = m$ and $W \perp a$.

Since $a_{i,j} \neq 0$ for some  $1<i<j$, a non-zero entry occurs after the $n^{th}$ entry of $a$. Thus, a suitable unit vector shows $U_{n} := \{0\}^n \oplus \RR^{m+1-n} \not \subset \{a\}^\perp$. Since $W \perp a$, the inclusion $W \cap U_{n}  \subset U_{n}$ is necessarily strict, which entails $\dim (W \cap U_{n}) \leq m-n$. Let $\pr: W \to \RR^n$ denote the projection onto the first $n$ coordinates.
By elementary linear algebra and since $\text{Ker}(\pr) = W \cap U_{n}$ by definition, $\dim(\text{Im}(\pr)) = \dim W - \dim(\text{Ker}(\pr)) \geq m - (m-n) = n$. Thus, $\pr(W) = \RR^n$ and the set $\{\pr((1,\chi^k))\}_{1\leq k \leq m} = \{(1,\chi^k_{1,2},\ldots ,\chi^k_{1,n})\}_{1 \leq k \leq m}$ contains $n$ linearly independent vectors, which we may assume to be indexed by $1 \leq k \leq n$ (reordering the $\chi^k$ if necessary).

Finally, we construct $\binom{n+1}{2}$ affinely independent solutions in $\tcf_{n+1}$ for the lifted equation 
\begin{align*}
a_{0} + a_{1,2}\chi_{1,2}+ \ldots +a_{n,n+1}\chi_{n,n+1} = 0, \quad \text{with}\quad  a_{i,n+1}= 0,\quad  1 \leq i \leq n.
\end{align*}
To simplify notation, assume that the new coordinates $\chi_{1,n+1},\ldots,\chi_{n,n+1}$ are added to the right of the previous coordinates $\chi_{1,2},\ldots,\chi_{n-1,n}$. We show that the $m+n=\binom{n+1}{2}$ points (recall $m := \binom{n}{2}$)
\begin{align*}
&\text{(a)} \quad (\chi^k,0,\ldots,0) \in \RR^{m+n}, \quad 1 \leq k \leq m, \quad \text{(with $n$ $0$'s added)},\\
&\text{(b)} \quad (\chi^k,\pr((1,\chi^k)))\in \RR^{m+n}, \quad 1 \leq k \leq n,
\end{align*}
solve the lifted equation, belong to $\tcf_{n+1}$ and are affinely independent.

The first statement follows from the choice of the $\chi^k$. The points in (a) belong to $\tcf_{n+1}$ by Lemma~\ref{lemma:zeroliftingvertices}. For (b), let $Y_{1},\ldots,Y_{n}$ be a stochastic model for $\chi^{k}$. Extend this model to $n+1$ variables $Y_{1},\ldots,Y_{n},Y_{n+1}$ by $Y_{n+1}:= Y_{1}$. Since $\pr((1,\chi^k)) = (1,\chi^k_{1,2},\ldots ,\chi^k_{1,n})$, this yields $(\chi^k,\pr((1,\chi^k))) \in \tcf_{n+1}$.

Linear independence of the $m+n$ points 
\begin{align*}
\{(1,\chi^k,0,\ldots,0)\}_{1 \leq k \leq m} \cup \{(1,\chi^k,\pr((1,\chi^k)))\}_{1 \leq k \leq n}
\end{align*} 
follows from the independence of $\pr((1,\chi^k)),\, 1 \leq k \leq n$ and the choice of the $\chi^k$.
\end{enumerate}
\end{proof}

\begin{remark}
By a slight abuse of notation, we will also call any vertex in the permutation orbit of $\chi^0$ a lifting of the vertex $\chi$ and any facet in the permutation orbit of a lifted facet a lifting of the respective facet.
\end{remark}

\subsection{\textbf{\upshape Unboundedness of denominators}}\label{sect:unbounded}

The following proposition shows that every rational number in the interval $[0,1]$ will appear as coordinate value in the vertex set of $\tcf_n$ starting from a sufficiently large $n$. The result is even sharper in that it detects a single vertex, whose coordinate values comprise a given finite subset of $[0,1]$-valued rational numbers.

\begin{proposition}[Unboundedness of denominators]\label{prop:unbounded}
$\phantom{a}$\\
For each finite subset $Q \subset \QQ \cap [0,1]$ of rational numbers in the interval $[0,1]$ there exists an $n \in \NN$ and a point $\chi \in \Ex(\tcf_n)$ whose coordinate-values $(\chi_{ij})_{1\leq i < j \leq n}$ include the set $Q$.\\
(By the lifting property, this holds for all $n' \geq n$, too.)
\end{proposition}

\begin{proof} By Lemma~\ref{lemma:unionsofvertices} it suffices to consider singletons $Q = \{q\}, q \in  \mathbb{Q} \cap [0,1]$. The proof only uses the following properties of $\chi \in \tcf_{n}$: 
\begin{itemize}
\item ``Positivity'' $\chi_{ij} \geq 0$ and 
\item the permutations of the valid inequalities 
\begin{align*}
\sum\limits_{i=1}^r\chi_{i,r+1} - \sum\limits_{1\leq i < j \leq r}\chi_{i,j} \leq 1, \quad r \geq 2
\end{align*}
which are hypermetric with $b$-vector $b =(1,\ldots,1,-1,0,\ldots,0)$ (with $r \geq 2$ times the entry 1), in particular permutations of the ``triangle inequality'' $\chi_{1,3}+\chi_{2,3}-\chi_{1,2} \leq 1$. The validity of these inequalities has been shown in Lemma~\ref{lemma:hypervalid}. 
\end{itemize}

The cases $q= 0$ and $q=1$ are trivial.

(I) We show that for rationals $q = \frac{1}{m}$ and $q = \frac{m-1}{m}$ it suffices to choose $n=2m+1$. Let $\Omega = \{\omega_{1}, \omega_{2,1},\ldots,\omega_{2,m},\omega_{3,1},\ldots,\omega_{3,m}\}$ be a set with  $2m+1$ elements
and define a positive function $g$ on $\Omega$ by 
\begin{align*}
g(\omega_{1})=\frac{1}{m}; \quad g(\omega_{2,i})=\frac{m-1}{m} \quad \text{and} \quad g(\omega_{3,i})=\frac{1}{m}, \quad  1 \leq i \leq m.
\end{align*} 
Normalizing $g$ by $c:=\frac{m^2+1}{m}$ yields a probability measure $\PP$ on $\Omega$ by $\PP(\{\omega\})=g(\omega)/c$. 
Now, we define $2m+1$ subsets of $\Omega$ as follows: 
\begin{align*}
A_{1,i}= \{\omega_{1},\omega_{2,i}\}, \quad A_{2,i}= \{\omega_{2,i},\omega_{3,i}\}, \quad 1 \leq i \leq m, \quad A_{3,1} = \{\omega_{3,1},\ldots,\omega_{3,m}\}.
\end{align*} 
Since all of these $2m+1$ sets have the same probability $1/c$, they define a point $\chi \in \tcf_{2m+1}$ as in Lemma~\ref{lemma:TCFevents}.

When viewed as an edge labelling $\chi$ can be described as follows: \\
Let $\{v_{1,1},\ldots,v_{1,m},v_{2,1},\ldots,v_{2,m},v_{3,1}\}$ denote the nodes of the support graph of $\chi$. A pair of nodes $v_{i_{1},i_{2}},v_{j_{1},j_{2}}$ is connected by an edge with label $\chi_{(i_{1},i_{2}),(j_{1},j_{2})} = \PP(A_{i_{1},i_{2}} \, | \,A_{j_{1},j_{2}})$. 
Draw the nodes  $\{v_{1,1},\ldots,v_{1,m}\}$ at the bottom level, they form a complete subgraph, all edges labelled by $\frac{1}{m}$. Above them draw the nodes $v_{2,1},\ldots,v_{2,m}$, where $v_{2,i}$ is connected to $v_{1,i}$ with an edge labelled $\frac{m-1}{m}$. Finally, the top node $v_{3,1}$ is connected to each $v_{2,1},\ldots,v_{2,m}$ with an edge labelled $\frac{1}{m}$. 

We show now that $\chi \in \Ex(\tcf_{2m+1})$. 
To this end, consider a representation $\chi = \lambda y +(1-\lambda) z$, $0 < \lambda < 1$, $y,z \in \tcf_{2m+1}$. 
Whenever $\chi$ satisfies a valid inequality as an equality, the same has to be true for $y$ and $z$. 
Consider $y$. All $\chi$-edges with label 0 have label 0 for $y$, too. Denote the unknown label $y_{(1,1),(2,1)}$ of the $y$-edge from $v_{1,1}$ to $v_{2,1}$ by $1-a \in [0,1]$. Note that $\chi$ satisfies a triangle inequality as an equality at $v_{2,1},v_{1,1},v_{1,2}$, since $\frac{m-1}{m} + \frac{1}{m} - 0 = 1$. This enforces $y_{(1,1),(1,2)} = a$. Now the triangle $v_{1,1},v_{1,2},v_{2,2}$ enforces $y_{(1,2),(2,2)} = 1-a$. Repeating this argument gives $y_{(1,i),(2,i)} = 1-a$ for all $1 \leq i \leq m$. From this, again just using triangles, it follows $y_{(1,i),(1,j)} = a$ for all $1 \leq i < j \leq m$ and $y_{(2,i),(3,1)} = a$ for all $1 \leq i \leq m$. Finally, observe that $\chi$ satisfies the hypermetric inequality given by $b=(0,\ldots,0,1,1,\ldots,1,-1)$, with $m$ 1's, as an equality $\sum_{i=1}^{m}\chi_{(3,1),(2,i)} - \sum_{1 \leq i <j \leq m}\chi_{(2,i),(2,j)} = m\cdot \frac{1}{m}  - 0 = 1$. Applied to $y$, this forces $m\cdot a = 1$, thus $a=1/m$. This shows $y=\chi$. The same argument applies to $z$. Hence $y= \chi = z$ and $\chi \in \Ex(\tcf_{2m+1})$.

(II) Now let $q = \frac{k}{m}$ for some $1 \leq k \leq m-1$. We modify the above construction to obtain a $\chi \in \Ex(\tcf_{2m+3})$ with some coordinate value equal to $q$. Extend $\Omega$ by two points to $\Omega' := \Omega \cup \{\omega_{3,m+1},\omega_{3,m+2}\}$. 
Extend $g$ by 
\begin{align*}
g(\omega_{3,m+1}) = \frac{k}{m} \quad \text{and} \quad
g(\omega_{3,m+2}) = \frac{m-k}{m}.
\end{align*}
Normalizing $g$ defines now $\PP'$. Use the same definitions for the sets $A_{i,j}$ as above and add the two sets 
\begin{align*}
A_{3,2} = \{\omega_{3,1},\ldots,\omega_{3,m-k},\omega_{3,m+1}\}
\quad \text{and} \quad A_{3,3}= \{\omega_{3,m+1},\omega_{3,m+2}\}.
\end{align*}
All sets have the same probability (the inverse of the normalizing constant) and thus, they define a point $\chi \in \tcf_{2m+3}$. 
Its support graph has two more nodes $v_{3,2},v_{3,3}$, corresponding to $A_{3,2}$ and $A_{3,3}$. 
The new edges are 
\begin{align*}
\chi_{(2,i),(3,2)}=\frac{1}{m}, \quad 1 \leq i \leq m-k, \quad  \chi_{(3,1),(3,2)} = \frac{m-k}{m}, \quad \chi_{(3,2),(3,3)} = \frac{k}{m}.
\end{align*}
Repeating the arguments from the first part shows $y = \chi$ on the ``old'' edges. Now, using the new triangles at $v_{3,2},v_{2,i},v_{1,i}$ for $1 \leq i \leq m-k$, we get 
\begin{align*}
y_{(2,i),(3,2)}=\frac{1}{m},\quad  1 \leq i \leq m-k.
\end{align*}
Note that a permutation of the hypermetric inequality $b=(1,\ldots,1,-1,0,\ldots,0)$ with  $m-k+1$ leading 1`s is fulfilled by $\chi$ as an equality, if the $-1$ corresponds to  $v_{3,2}$ and the 1's correspond to $v_{2,1},\ldots,v_{2,m-k},v_{3,3}$. Applied to $y$, this yields $(m-k)\cdot \frac{1}{m} + y_{(3,2),(3,3)} = 1$, thus $y_{(3,2),(3,3)} = \frac{k}{m}$. Finally, the triangle at $v_{3,1},v_{3,2},v_{3,3}$ implies $y_{(3,1),(3,2)} = \frac{m-k}{m}$. Thus, $y=\chi$ and the same argument applies to $z$. Hence, $\chi \in \Ex(\tcf_{2m+3})$.
\end{proof}
For $n \leq 4$ we have seen that $\cpp_{n} = \tcf_{n}$ (Proposition \ref{prop:cpplowdim}). This is complemented by the following result.
\begin{corollary}\label{cor:nocpp}
$\phantom{a}$\\
For $n \geq 5$ we have $\Ex(\tcf_n) \not \subset \{0,1\}^{E_n}$ and, in particular, $\cpp_n \subsetneq \tcf_n$.
\end{corollary}

\begin{proof}
By the lifting of extremal points (Lemma~\ref{lemma:zeroliftingvertices}) it suffices to prove this for $n=5$. For $q = \frac{1}{2}$ the construction (I) in the proof of Proposition~\ref{prop:unbounded} yields an example with $n=5$.
\end{proof}

\begin{remark}
For $q = \frac{1}{2}$ the above construction (I) is optimal: it gives the smallest possible  $n$ for the occurence of $q$ as the coordinate value of a vertex of $\tcf_{n}$. To realize $q=\frac{1}{3}$ the construction (I) uses $n=7$, but a coordinate value $\frac{1}{3}$ already occurs for $n=6$, as the computation of $\Ex(\tcf_{6})$ in Section~\ref{sect:compresults} shows. 
\end{remark}

\subsection{\textbf{\upshape Non-hypermetric facets of $\tcf_{n}$ for $n \geq 6$}}\label{sect:nonhypfacets}


We give a proof for the existence of non-hypermetric facets. 
First, we provide two simple necessary conditions for hypermetricity. Of course, multiplying a given (affine) inequality by some constant $q \neq 0$ does not change the halfspace it describes. Thus, one is often interested, if a given inequality is hypermetric up to a suitable multiplication.

\begin{lemma}
\label{lemma:condforhypermetricity}
Suppose that an inequality $\sum_{1\leq i < j \leq n} c_{ij}x_{ij} \leq c_{0}$ (with rational coefficients) is equivalent to a hypermetric inequality, i.e., it becomes a hypermetric inequality defined by some $b \in \ZZ^n$ after multiplication with a suitable constant $q \in \QQ \setminus \{0\}$. Then we have:
\begin{enumerate}[a)]
\item The edges $\{i,j\} \subset E_n$ with $c_{ij} \neq 0$ form a complete subgraph of the support graph $K_n$.
\item The vectors $v_{1}:=(c_{1,3},\ldots,c_{1,n})$ and $v_{2}:=(c_{2,3},\ldots,c_{2,n})$ are linearly dependent.  
\end{enumerate}
\end{lemma}
\begin{proof}
\begin{enumerate}[a)]
\item By assumption $c_{ij} = -q^{-1}\cdot b_{i}b_{j}$ for some $q \in \QQ \setminus \{0\}$. Thus, the non-zero $c_{ij}$ correspond to the edges of the complete subgraph with nodes $\{1\leq i \leq n \, | \, b_{i} \neq 0\}$.
\item Again, $c_{ij} = -q^{-1}\cdot b_{i}b_{j}$. If $b_{2}=0$, then $v_{2} = 0$, thus, $v_{1},v_{2}$ are dependent. If $b_{2}\neq 0$, then $v_{1} = (b_{1}/b_{2})\cdot v_{2}$. 
\end{enumerate}
\end{proof}

\begin{remark}\label{rk:condforhyp}
Note that criterion a) of Lemma~\ref{lemma:condforhypermetricity} also implies: if there is at least one 0-coefficient, there have to be at least $n$ 0-coefficients, and if the first $n-1$ coefficients $c_{1,2},\ldots,c_{1,n}$ are positive, all have to be positive.
\end{remark}

\noindent The following proposition shows the existence of non-hypermetric facets of $\tcf_n$ starting from $n \geq  6$.
It was inspired by the 2nd inequality of Generator 7 in Table~\ref{table:TCFsixfacets}.

\begin{proposition}[Non-hypermetric facets of $\tcf_n$ for $n \geq  6$]
\label{prop:nonhypfacets} 
$\phantom{a}$\\
For $n \geq 6$ there are non-hypermetric facets of $\tcf_{n}$.\\ An example, for arbitrary $n\geq 6$, is given by the facet inducing inequality 
\begin{align*}
\sum_{i=1}^{5}x_{i,6} - \sum_{i=1}^{4}x_{i,i+1} - x_{1,5}\leq 2.
\end{align*}
\end{proposition}

\begin{proof} 
By the lifting of facets (Proposition~\ref{prop:zeroliftingfacets}), it suffices to consider the case $n = 6$. We start with a simple observation for 0-1-vectors of even length: For $y \in \{0,1\}^{2k}, k \in \NN,$ the inequality
\begin{align}\label{eqn:cyclicinequality}
\sum_{i=1}^{2k-1}y_{i}\cdot (y_{2k}-y_{\pi(i)}) \leq (k-1)\cdot y_{2k}
\end{align} 
holds, where $\pi$  is the cyclic permutation of $1,\ldots,2k-1$, i.e., $\pi(i)=i+1$, $i < 2k-1$ and $\pi(2k-1)=1$. The observation is trivial if $y_{2k} = 0$. To handle the case $y_{2k} = 1$ observe that $y_{i}(1-y_{\pi(i)}) = 1$ if and only if $y_{i}=1$ and $y_{\pi(i)}=0$. There can be at most $k-1$ occurrences of the word ``10'' in the string $y_{1},\ldots,y_{2k-1},y_{1}$.
Applying (\ref{eqn:cyclicinequality}) to arbitrary binary random variables $Y_{1},\ldots,Y_{2k}$ and taking expectations yields
\begin{align*}
\sum_{i=1}^{2k-1} \EE(Y_{i}Y_{2k}) - \sum_{i=1}^{2k-1}\EE(Y_{i}Y_{\pi(i)}) \leq (k-1) \EE(Y_{2k}).
\end{align*}
If, additionally, $a:=\EE(Y_{1})=\ldots=\EE(Y_{2k}) > 0$, dividing by $a$ gives the following valid inequality for $\tcf_{2k}$, where $x_{i,j} := \frac{1}{a}\EE(Y_iY_j)$,
\begin{align}\label{eqn:nonhypinequality}
\sum_{i=1}^{2k-1}x_{i,2k} - \sum_{i=1}^{2k-1}x_{i,\pi(i)} \leq (k-1)
\end{align}
(which has a very simple supporting graph when we identify $x_{2k-1,1}$ with $x_{1,2k-1}$).
Assume now $k \geq 3$. Since the coefficients of $x_{1,2}$ and $x_{2,3}$ are $-1$ and the coefficient of $x_{1,3}$ is $0$, the non-zero coefficients do not define a complete subgraph of the support graph.  Thus, Lemma~\ref{lemma:condforhypermetricity} a) shows that the above  inequality is not hypermetric  for $k\geq3$.

Finally, we show that for $k=3$, the inequality (\ref{eqn:nonhypinequality}) 
defines a facet for $\tcf_{6} \subset \RR^{E_6}$: To this end, we define $\lvert E_6\rvert=15$ points $x^r,y^r,z^r \in \{0,1\}^{E_{6}}$, $1 \leq r \leq 5$ by
\begin{align*}
&(a) \quad x^r_{i,j}=1 \quad :\Leftrightarrow \quad  \{i,j\} \subset A_{r} :=\{r,\pi^2(r),6\},\\ 
&(b) \quad y^r_{i,j}=1 \quad :\Leftrightarrow \quad \{i,j\} \subset B_{r} := \{r,\pi(r),\pi^3(r),6\},\\
&(c) \quad z^r_{i,j}=1 \quad :\Leftrightarrow \quad  (\{i,j\} \subset B_{r} \quad \text{or} \quad \{i,j\} = \{\pi^2(r),\pi^4(r)\}).
\end{align*}
Note that these points are clique partion points and thus belong to the set $\tcf_6$ by Proposition~\ref{prop:cpp}.
Using the support graph of (\ref{eqn:nonhypinequality})  for $k=3$, it can be easily seen that they solve (\ref{eqn:nonhypinequality}) for $k=3$ as an equality. Moreover, these 15 points are affinely independent, since they are even linearly independent as the determinant of the corresponding $15 \times 15$ $0$-$1$-matrix is $-2\neq 0$.
\end{proof}



\subsection{\textbf{\upshape Embedding $\tcf_{n}$ into the Correlation and Cut polytopes}}\label{sect:corcut}

We saw already in the proof of Corollary~\ref{cor:finitepolytope} that the polytope $\tcf_n$ can be viewed essentially as a \emph{projection} of the convex polytope $\Theta_n$ onto several coordinates as in (\ref{eqn:psin}). In this section we will see that the polytope $\tcf_n$ can be embedded into the so-called \emph{correlation polytope} (or, equivalently, the so-called \emph{cut polytope}, see Proposition~\ref{prop:TCFembedding} below). Thereby, we obtain a ``dual'' description of $\tcf_n$ as the \emph{intersection} of a polytope with an affine subspace.

To this end, we need to review some notation and results from \cite{dezalaurent_97}.
Remember that  $E_n$ denotes  the set of edges of the complete graph $K_n$ with vertices $V_n=\{1,\dots,n\}$. For $R \subset V_n$ we define a \emph{correlation vector} $\pi(R) \in \{0,1\}^{V_n \cup E_n}$ by 
\begin{align*}
\pi(R)_i = \Eins_{i \in R}, \quad  1\leq i \leq n \qquad \text{and} \qquad  \pi(R)_{ij} = \Eins_{i \in R} \Eins_{j \in R}, \quad 1 \leq i < j \leq n. 
\end{align*}
The \emph{correlation polytope} is then defined as the convex hull of these $2^n$ correlation vectors in $\RR^{V_n \cup E_n}$
\begin{align*}
\cor_n := \conv\left(\left\{\pi(R) \,:\, R \subset V_n\right\}\right).
\end{align*}

\begin{lemma}[\cite{dezalaurent_97} Prop.\ 5.3.4] \label{lemma:CORevents}
$\phantom{a}$\\
A point $p \in \RR^{V_n\cup E_n}$ belongs to $\cor_n$ if and only if it can be written as  $p_i = \PP(A_i)$, $1\leq i \leq n$ and $p_{ij} = \PP(A_i \cap A_j)$, $1\leq i < j \leq n$ for some probability space $(\Omega,\A,\PP)$ and measurable subsets $A_1,\dots,A_n \in \A$.
\end{lemma}

\noindent Secondly, let $S \subset V_{n+1}$. A \emph{cut vector} $\delta(S) \in \{0,1\}^{E_{n+1}}$ is defined through
\begin{align*}
\delta(S)_{ij} = \Eins_{\lvert S \cap \{i,j\} \rvert = 1}, \qquad 1 \leq i < j \leq n+1.
\end{align*}
Since $\delta(S)=\delta(S^c)$, there are, in fact, $2^{n+1}/2= 2^n$ different points $\delta(S)$. The \emph{cut polytope} is defined as the convex hull of these cut vectors in $\RR^{E_{n+1}}$
\begin{align*}
\cut_{n+1} :=  \conv\left(\left\{\delta(S) \,:\, S \subset V_{n+1}\right\}\right).
\end{align*}
Being $\{0,1\}$-valued, the correlation vectors and the cut vectors are automatically the extremal points of their convex hulls
\begin{align*}
\Ex\left(\cor_n\right)=\{ \pi(R) \,:\, R \subset V_n \} \quad \text{and} \quad
\Ex\left(\cut_{n+1}\right)=\{ \delta(S) \,:\, S \subset V_{n+1} \}.
\end{align*}
It is a well-known result that $\cor_n \subset \RR^{V_n \cup E_n}$ and $\cut_{n+1} \subset \RR^{E_{n+1}}$ can be transformed into each other by a linear bijection. 

\begin{proposition}[\cite{dezalaurent_97}, Section~5.2)] \label{prop:covariancemapping}
$\phantom{a}$\\
The {covariance mapping} $\zeta_n: \RR^{V_n \cup E_n}  \rightarrow \RR^{E_{n+1}}$,
which maps $p \in \RR^{V_n \cup E_n}$ to $\zeta_n(p)=x \in \RR^{E_{n+1}}$ via 
\begin{align*}
x_{i,n+1} = p_i, \quad 1 \leq i \leq n \qquad \text{and} \qquad  x_{ij}=p_i + p_j - 2p_{ij}, \quad 1 \leq i < j \leq n, 
\end{align*}
induces a linear bijection 
\begin{align*}
\zeta_n:\cor_{n} \rightarrow \cut_{n+1}.
\end{align*}
\end{proposition}

\begin{remark}
In \citep{dezalaurent_97} the inverse $\xi_n:=\zeta_n^{-1}$ is termed covariance mapping. For us, it was more instructive to work with $\zeta_n$ instead of $\xi_n$.
\end{remark}

A probabilistic description of $\cut_{n+1}$ is as follows. 
Here the symmetric difference between sets $A$ and $B$ will be denoted by $A \triangle B = (A \setminus B) \cup (B \setminus A)$.

\begin{lemma} \label{lemma:CUTevents}
A point $x \in \RR^{E_{n+1}}$ belongs to the cut polytope $\cut_{n+1}$ if and only if one of the following two equivalent statements holds true:
\begin{enumerate}[(i)]
\item $x_{i,n+1} = \PP(A_i)$, $1 \leq i \leq n$ and $x_{ij} = \PP(A_i \triangle A_j)$, $1 \leq i < j \leq n$ for some probability space $(\Omega,\A,\PP)$ and measurable subsets $A_1,\dots,A_n \in \A$.
\item $x_{ij} = \PP(B_i \triangle B_j)$, $1 \leq i < j \leq n+1$ for some probability space $(\Omega,\A,\PP)$ and measurable subsets $B_1,\dots,B_{n+1} \in \A$.
\end{enumerate}
\end{lemma}

\begin{proof} The equivalence to (i) is an immediate consequence of Lemma~\ref{lemma:CORevents} and Proposition~\ref{prop:covariancemapping}. The equivalence of (i) and (ii) can be seen as follows: (i) $\Rightarrow$ (ii): Set $B_i=A_i$, $1 \leq i \leq n$ and $B_{n+1}=\emptyset$. (ii) $\Rightarrow$ (i): Set $A_i=B_i \triangle B_{n+1}$, $1\leq i \leq n$ and use that $(C \triangle D) \triangle ( E \triangle D)= C \triangle E$ for any triplet of sets $C,D,E$.
\end{proof}

\noindent Finally, this enables us to interpret $\tcf_n$ as an intersection of $\cor_n$ (\resp $\cut_{n+1}$) with an affine subspace of $\RR^{V_n \cup E_n}$ (\resp $\RR^{E_{n+1}}$) in the following sense.

\begin{proposition}[Embedding $\tcf_n$ into the correlation polytope]\label{prop:TCFembedding}
$\phantom{a}$\\
The injective affine map $\iota_n:\RR^{E_n} \rightarrow \RR^{V_n \cup E_n}$ which maps $\chi \in \RR^{E_n}$ to $\iota_n(\chi)=p \in \RR^{V_n \cup E_n}$ via
\begin{align*}
p_i=\frac{1}{n}, \quad 1 \leq i \leq n \qquad \text{and} \qquad  p_{ij}= \frac{\chi_{ij}}{n}, \quad 1 \leq i < j \leq n,
\end{align*}
induces a bijection
\begin{align*}
\iota_n:\tcf_{n} \rightarrow \cor_{n} \cap \left\{ p \in \RR^{V_n \cup E_n} \,:\, p_i=\frac{1}{n}, \, i=1,\dots,n\right\}.
\end{align*}
\end{proposition}

\begin{proof}
The map $\iota_n$ is injective by definition. First, we show that $\iota_n(\tcf_n) \subset \cor_n$. Because of Lemma~\ref{lemma:TCFevents} and Remark~\ref{remark:TCFevents}, a point $\chi \in \tcf_n$ has a stochastic model $A_1,\dots,A_n$ with $\PP(A_1)=\dots =\PP(A_n)=1/n$ and $\chi_{ij}=\PP(A_i \cap A_j)/\PP(A_j)$.  Lemma~\ref{lemma:CORevents}, applied to $A_{1},\ldots,A_{n}$ and $\PP$, shows that $\iota_{n}$ maps $\tcf_n$ to $\cor_{n}$. Now, suppose that $p \in \cor_{n}\cap \bigcap_{i=1}^n\{p_{i} = 1/n\}$. By Lemma~\ref{lemma:CORevents} there is a stochastic model with sets $A_{1},\ldots,A_{n}$, $\PP(A_{1})=\ldots=\PP(A_{n}) = 1/n$, $\PP(A_{i}\cap A_{j}) = p_{ij}$. Thus, $\chi = (n\cdot p_{ij})_{1 \leq i < j \leq n} $ is a preimage of $p$ in $\tcf_{n}$.
\end{proof}

\noindent Note that we just established the following equivalences
\begin{align*}
\chi \in \tcf_n \quad \Leftrightarrow \quad \iota_n(\chi) \in \cor_n \quad \Leftrightarrow \quad \zeta_n \circ \iota_n (\chi) \in \cut_{n+1}.
\end{align*}
In particular, one can pull back facets from $\cut_{n+1}$ to $\cor_{n}$ with the covariance mapping $\zeta_{n}$, and further, we obtain an $\mathcal{H}$-representation for $\tcf_{n}$ using $\zeta_{n} \circ \iota_{n}$. Thus, any $\mathcal{H}$-representation of $\cor_n$ or $\cut_{n+1}$ yields an $\mathcal{H}$-representation of $\tcf_n$ as follows. %

\begin{proposition}[Pulling back $\mathcal{H}$-representations]\label{prop:cutcorfacets}
$\phantom{a}$
\begin{enumerate}[a)]
\item {\normalfont (\cite{dezalaurent_97} Prop.\ 26.1.1, p.~402)} 
$\phantom{a}$\\
The covariance mapping $\xi_n:=\zeta_n^{-1}$ maps a valid inequality for $\cut_{n+1}$ (resp.\ facet of $\cut_{n+1}$)
\begin{align}\label{eqn:cuthalfsp}
\sum_{1 \leq i < j \leq n+1} c_{ij}x_{ij} \leq c_{0}
\end{align}
to the following valid inequality $\cor_{n}$ (resp.\ facet of $\cor_{n}$)
\begin{align}\label{eqn:corhalfsp}
\sum_{1\leq i \leq n} b_{i}p_{i} + \sum_{1 \leq i < j \leq n} (-2c_{ij})p_{ij} \leq c_{0} 
\quad \text{with} \quad 
b_{i} = \sum_{1 \leq s < i}c_{si} + \sum_{i < s \leq n+1}c_{is}.
\end{align}
\item The above valid inequality (resp.\ facet) of $\cut_{n+1}$ induces the following valid inequality for $\tcf_{n}$ via 
$\zeta_n \circ \iota_n$
\begin{align}\label{eqn:tcfhalfsp}
\sum_{1\leq i < j \leq n} (-2c_{ij}) \chi_{ij} \leq n\cdot c_{0} - 2 \sum_{1\leq i < j \leq n} c_{ij} - \sum_{i=1}^n c_{i,n+1}.
\end{align}
If applied to all elements of an $\mathcal{H}$-representation of $\cut_{n+1}$ (e.g.\ all facets of $\cut_{n+1}$), this gives an $\mathcal{H}$-representation for $\tcf_{n}$.
\end{enumerate}
\end{proposition}

\begin{proof}
\begin{enumerate}[b)]
\item
It suffices to replace $x_{ij}$ in Inequality~(\ref{eqn:cuthalfsp}) by 
\begin{align*}
(\zeta_n \circ \iota_n(\chi))_{ij}
= \left\{\begin{array}{ll}
\frac{1}{n} & \qquad j=n+1,\smallskip\\
\frac{2}{n} - \frac{2}{n} \chi_{ij} & \qquad 1 \leq i < j \leq n.
\end{array}\right.
\end{align*}
and to reorder the resulting terms.
\end{enumerate}
\end{proof}

\paragraph{\textbf{\upshape Dual views on $\tcf_n$}} Summarizing, we obtain two complementary views on the polytope $\tcf_n$ which may be illustrated as follows.
\begin{center}
\begin{minipage}{0.45\textwidth}
\xymatrix{
  \Theta_n \ar@{>>}[d]_{\psi_n} \\
  \tcf_n  \ar@{^{(}->}[r]_{\iota_n}  & \cor_n  \ar[r]_{\zeta_n} & \cut_{n+1} 
}
\end{minipage}
\hspace{2mm}
\begin{minipage}{0.45\textwidth}
\xymatrix{
  \RR^{\finite_n^{(2)}} \ar@{>>}[d]_{\psi_n} \\
  \RR^{E_n}  \ar@{^{(}->}[r]_{\iota_n}  & \RR^{V_n \cup E_n}  \ar[r]_{\zeta_n} & \RR^{E_{n+1}} 
}
\end{minipage}
\end{center}
Here $\psi_n$ is given by the ``projection'' map (\ref{eqn:psin}), the map $\iota_n$ is the embedding from Proposition~\ref{prop:TCFembedding} and $\zeta_n$ the covariance mapping from Proposition~\ref{prop:covariancemapping}. 
While any $\mathcal{V}$-representation of $\Theta_n$ easily yields a $\mathcal{V}$-representation of $\tcf_n$ essentially by a projection, any $\mathcal{H}$-representation of $\cut_{n+1}$ easily yields an $\mathcal{H}$-representation of $\tcf_n$ essentially by an intersection. Unfortunately, $\Theta_n$ is a priori given by its facets (an $\mathcal{H}$-represenation), while $\cut_{n+1}$ is a priori given by its vertices (a $\mathcal{V}$-representation) and not the other way around, such that both views come along with certain drawbacks. 
At least the facets of $\cut_{n+1}$ are classified to some extent.

\paragraph{\textbf{\upshape The facets of $\cut_{n+1}$ and their generators}} \citep[Part V]{dezalaurent_97} 
Let us consider the following two kinds of actions on $\RR^{E_{n+1}}$. On the one hand the symmetric group $S_{n+1}$ acts on $\RR^{E_{n+1}}$ by node permutations: $(\sigma(x))_{ij}:= x_{\sigma(i)\sigma(j)}$ for $\sigma \in S_{n+1}$.  These actions are simply called \emph{permutations}. On the other hand each of the $2^n$ cut vectors $\delta(S)$ acts on $\RR^{E_{n+1}}$ by
\begin{align*}
(\delta(S)(x))_{ij} 
= \left\{ 
\begin{array}{ll}
1-x_{ij} & \text{if } \delta(S)_{ij}=1,\\
x_{ij} & \text{otherwise},
\end{array}
\right.
\end{align*}
for any $S \subset V_{n+1}=\{1,\dots,n+1\}$, \ie coordinates $x_{ij}$ corresponding to the edges of the cut beween $S$ and $S^c$ are replaced by $1-x_{ij}$. These actions are called \emph{switchings}. Note that $\delta(S) \circ \delta(R) = \delta(S \triangle R)$ and that $\delta(S) \circ \sigma = \sigma \circ \delta(\sigma(S))$.
In fact, both kinds of actions can be restricted to the cut polytope $\cut_{n+1}$. For any $\sigma \in S_{n+1}$ and any $S \subset V_{n+1}$
\begin{align*}
\sigma(x) \in \cut_{n+1} \quad \Leftrightarrow \quad x \in \cut_{n+1} \quad \Leftrightarrow \quad \delta(S)(x) \in \cut_{n+1}.
\end{align*}
These permutations and switchings on the polytope $\cut_{n+1}$ induce, of course, corresponding actions on its facets.
First, it is not surprising that (\ref{eqn:cuthalfsp}) is a facet inducing inequality of $\cut_{n+1}$ if and only if
\begin{align*}
\sum_{1 \leq i < j \leq n+1} c_{\sigma(i)\sigma(j)}x_{ij} \leq c_{0}
\end{align*}
is facet inducing for $\cut_{n+1}$. Second, any facet inducing inequality (\ref{eqn:cuthalfsp}) can be \emph{switched} by a cut vector $\delta(S)$
to another facet inducing inequality of $\cut_{n+1}$ which is given by
\begin{align*}
\sum_{1 \leq i < j \leq n+1} (1-2\delta(S)_{ij})c_{ij}x_{ij} \leq c_{0} - \sum_{1 \leq i < j \leq n+1} \delta(S)_{ij}c_{ij}.
\end{align*}
Let $O^{SP}(g,c_{0})$ denote the full orbit of a facet $g(x) \leq c_{0}$ under all possible finite applications of switchings and permutations to $g(x) \leq c_{0}$. The set of all facets of $\cut_{n+1}$ splits into finitely many such orbits, say  $O^{SP}_{i}$, $i\in I$. Choosing one facet $g^{(i)}(x) \leq c^{(i)}_{0}$ from each orbit $O^{SP}_{i}$ yields a set  of representatives $g^{(i)}(x) \leq c^{(i)}_{0}$, $i\in I$, of the facets of $\cut_{n+1}$, up to switchings and permutations. In this way \emph{generators} for the facets of $\cut_{n+1}$  are given in the literature. It is a feature of the cut polytope that it always has a set of \emph{homogeneous generators}, i.e. with $c^{(i)}_{0} = 0$, $i \in I$ \citep[Section 26.3.2]{dezalaurent_97}.

The facets of $\cut_{n+1}$ and corresponding generators are known for $n \leq 7$ \citep[p.~504]{dezalaurent_97}. In Table~\ref{table:cutgenerators} (Appendix~\ref{sect:tables}) we list the 11 generators of the $116\phantom{.}764$ facets of $\cut_7$ that will be used to derive the facets of $\tcf_6$.

\paragraph{\textbf{\upshape Relations to unit covariances}}
In their works on {\citeauthor{mcmillan1955history}'s} ({\citeyear{mcmillan1955history}}) realization problem concerning covariances of binary processes
{\cite{quint08}, \cite{la13,la15} and \cite{shepp1963positive,shepp1967covariances}} considered $\{-1,1\}$-valued random vectors $(U_1,\dots,U_n)$ (instead of $\{0,1\}$-valued vectors) and studied the set of \emph{unit covariances} 
\begin{align*}
\mathcal{U}_n := \left\{ u \in \RR^{E_n}  \pmid
\begin{array}{l}
\text{$u_{ij}=\EE (U_iU_j)$ where}\\
\text{$U_1,\dots,U_n$ take values in $\{-1,1\}$} 
\end{array}
\right\}.
\end{align*}
As a consequence of Lemma~\ref{lemma:CUTevents} (ii) (set $B_i=\{U_i=1\}$ therein) the cut polytope $\cut_n$ and the set of unit covariances $\mathcal{U}_n$ are affine equivalent via the bijective mapping $g_n:\RR^{E_n} \rightarrow \RR^{E_n}$, $g_n(x)=\frac{1}{2}(1-x)$ through
\begin{align}\label{eq:cutunit}
\cut_n = g_n(\mathcal{U}_n).
\end{align}
Let us further denote for $c \in [0,1]$ as in {\cite{shepp1963positive}}
\begin{align*}
\mathcal{U}_n(c) := \left\{ u \in \RR^{E_n}  \pmid
\begin{array}{l}
\text{$u_{ij}=\EE (U_iU_j)$ where}\\
\text{$U_1,\dots,U_n$ take values in $\{-1,1\}$} \\
\text{and $\PP(U_1=1)=\dots=\PP(U_n=1)=c$} 
\end{array}
\right\}.
\end{align*}
It is immediate that $\mathcal{U}_n(c)=\mathcal{U}_n(1-c)$ and 
repeating an argument from {\cite{shepp1963positive}}, p.~10, 
it is not difficult to see that 
$\mathcal{U}_n(c)$, $0\leq c \leq 1/2$ are increasing towards $\mathcal{U}_n(1/2)=\mathcal{U}_n$. 
The latter equality follows from the fact that the unit covariance of a $\{-1,1\}$-valued random vector remains unchanged after multiplication with an independent $\{-1,1\}$-valued zero mean variable.
The affine equivalence \eqref{eq:cutunit} can be refined to 
\begin{align}\label{eq:cutunitrefined}
\cut_n(c) = g_n(\mathcal{U}_n(c)), \quad c \in [0,1]
\end{align} 
if we set
\begin{align*}
\cut_n(c):=\pr_n(\cut_{n+1} \cap \{x \in \RR^{E_{n+1}} \,:\, x_{i,n+1}=c, \, 1\leq i \leq n\})
\end{align*}
and $\pr_n:\RR^{E_{n+1}} \rightarrow \RR^{E_n}$ is the projection onto the edges not containing the vertex $n+1$. 
A probabilistic description of the polytopes $\cut_n(c)$, $c \in[0,1]$ 
follows from the equivalence (i) in Lemma~\ref{lemma:CUTevents} (set $A_i=\{U_i=1\}$) and thereby proves the refinement \eqref{eq:cutunitrefined} as follows.

\begin{lemma}\label{lemma:CUTCevents}
A point $x \in \RR^{E_n}$ belongs to $\cut_n(c)$ if and only if can be written as $x_{ij}=\PP(A_i \triangle A_j)$, $1 \leq i < j \leq n$ for some probability space $(\Omega,\A,\PP)$ and measurable subsets $A_1,\dots,A_n \in \A$ satisfying $\PP(A_i)=c$, $1 \leq i \leq n$.
\end{lemma}

A direct connection of unit covariances to $\tcf_n$ can be obtained 
from Lemma~\ref{lemma:TCFevents} and Remark~\ref{remark:TCFevents} (set $U_i=2 \cdot \eins_{A_i}-1$ therein) as
\begin{align*}f_n(\tcf_n) = \mathcal{U}_n(1/n),\end{align*}
where $f_n:\RR^{E_n} \rightarrow \RR^{E_n}$ is the bijective affine mapping $f_n(x)=\frac{4}{n}x-\frac{4}{n}+1$.
It can be easily checked that the following diagram commutes
if $\zeta_n$ and $\iota_n$ are the respective affine mappings from Propositions~\ref{prop:covariancemapping} and \ref{prop:TCFembedding}.

\begin{center}
\begin{minipage}{0.45\textwidth}
\xymatrix{
  \tcf_n \ar@{^{(}->}[r]^{\zeta_n \circ \iota_n} \ar[d]_{f_n} & \cut_{n+1} \ar@{>>}[d]^{\pr_n}\\
  \mathcal{U}_n(1/n)  \ar[r]_{g_n}  & \cut_n(1/n) 
}
\end{minipage}
\hspace{2mm}
\begin{minipage}{0.45\textwidth}
\xymatrix{
  \RR^{E_{n}} \ar@{^{(}->}[r]^{\zeta_n \circ \iota_n} \ar[d]_{f_n} & \RR^{E_{n+1}} \ar@{>>}[d]^{\pr_n}\\
  \RR^{E_{n}}  \ar[r]_{g_n}  & \RR^{E_{n}}
}
\end{minipage}
\end{center}

We remark the simple form of the mapping $(g_n \circ f_n)(x)=\frac{2}{n}(1-x)$. The following lemma shows that the polytopes $\cut_n(c)$, $0 < c \leq 1/n$ and  $\mathcal{U}_n(c)$, $0 < c \leq 1/n$ are also affine isomorphic.

\begin{lemma} For $\lambda \in [0,1]$ we have 
\begin{align*}
\cut_n(\lambda/n)&=\lambda \cdot \cut_n(1/n)\\
\mathcal{U}_n(\lambda/n)&=\lambda \cdot \mathcal{U}_n(1/n) + (1-\lambda).
\end{align*}
\end{lemma}

\begin{proof}
The second relation follows from the first by applying the map $g_n$. We prove the first statement using Lemma~\ref{lemma:CUTCevents} for both inclusions (``$\subset$'' and ``$\supset$''), where we may assume that $A := \bigcup_{i=1}^n A_i \neq \Omega$ (otherwise add a point to $\Omega$). 
An element $x \in \cut_n(\lambda/n)$ admits the representation $x_{ij}=\PP(A_i \triangle A_j)$ for sets $A_1,\dots,A_n$ with $\PP(A_i)=\lambda/n$. 
It follows that $\PP(A)\leq \lambda$ and we can extend $\PP':=(1/\lambda) \cdot \left.\PP\right|_A$ to a probability measure on $\Omega$ which gives $x_{ij}=\lambda \PP'(A_i\triangle A_j)$ with $(\PP'(A_i\triangle A_j))_{1 \leq i<j \leq n} \in \cut_n(1/n)$. 
Conversely, $x \in \cut_n(1/n)$ admits the representation $x_{ij}=\PP(A_i \triangle A_j)$ for sets $A_1,\dots,A_n$ with $\PP(A_i)=1/n$ and we can extend the measure $\PP':= \lambda \cdot \left.\PP\right|_A$ to $\Omega$ which gives 
$\lambda \cdot x_{ij}= \PP'(A_i\triangle A_j)$ with $(\PP'(A_i\triangle A_j))_{1 \leq i<j \leq n} \in \cut_n(\lambda/n)$.
\end{proof}

Together with $(g_n \circ f_n)^{-1}(y)=1-\frac{n}{2}y$ this identifies the polytope $\tcf_n$ as
\begin{align}\label{eq:tcfcutunit}
\tcf_n = 1 - \frac{1}{2c} \cut_n(c) = 1- \frac{1}{4c} \left(1-\mathcal{U}_n(c)\right), \quad \text{ for any } c \in (0,1/n].
\end{align}
Hence, any better understanding on one of the polytopes in \eqref{eq:tcfcutunit}
will automatically transfer to all the other ones.



\section{Computational results}\label{sect:compresults}
We computed the vertices and facets of $\tcf_n$ for $n \leq 6$  using the software \texttt{R} \citep{R} and \texttt{polymake} \citep{polymake}. 
Their explicit representatives are documented in the tables of Appendix~\ref{sect:tables}.
In order to obtain the vertices and facets of $\tcf_6$,
we had to use both views on $\tcf_{6}$ described at the end of Section~\ref{sect:corcut}: a $\mathcal{V}$-representation of $\tcf_6$ was obtained via the polytope $\Theta_{6}$, the reduction to the vertex representation Ex($\tcf_{6}$) took extra efforts. An $\mathcal{H}$-representation for $\tcf_{6}$ was obtained via the embedding into $\cut_{7}$ (using the known facet-representation), from which we extracted a facet-representation of $\tcf_{6}$ using the previously computed vertices Ex($\tcf_{6}$). Below we  give a detailed description of our methods.



\paragraph{\textbf{\upshape The vertices and facets of $\tcf_n$ for $ 2 \leq  n \leq 5$}}
$\phantom{a}$\\
For $n \leq 4$ the vertices and facets of $\tcf_n$ were computed already in \cite{strokorb_13} p.~62. In particular, all vertices are $\{0,1\}$-valued, hence clique partition points (cf.\ Proposition~\ref{prop:cpplowdim}). 

The vertices and facets of $\tcf_5$ have been obtained directly using \texttt{R} and \texttt{polymake} via the two different approaches presented in Section~\ref{sect:corcut} (leading to the same result): via the polytope $\Theta_5$ and the embedding into the correlation polytope $\cor_5$ (defined by its vertices).
Here, for $n=5$, the software \texttt{R} was simply used to generate the input for \texttt{polymake}.  From these computations we see that $\tcf_5$ has 214 vertices in 11 permutation orbits as listed in Table~\ref{table:TCFfivevertices} (Appendix~\ref{sect:tables}). While 52 vertices in 7 permutation orbits are $\{0,1\}$-valued (the expected clique partition points), for the first time also $\{0,\OH\}$-valued vertices turn up (162 in 4 permutation orbits).

Representatives for the permutation orbits of the facets of $\tcf_{n}$ for each $2\leq n \leq 5$ are listed in Table~\ref{table:hypermetricTCFfacets} in the Appendix~\ref{sect:tables}. Since all facets turned out to be hypermetric, we  describe them by their defining vectors $b \in \ZZ^n$.  In particular, we obtain the following result.

\begin{proposition}\label{prop:nleq5hypermetric} For $n \leq 5$ all facets of $\tcf_{n}$ are hypermetric.
\end{proposition}


\noindent Let us now turn to the case $n=6$, which needed additional arguments to reduce the computational burden.

\paragraph{\textbf{\upshape The vertices of $\tcf_6$}}
$\phantom{a}$\\
According to our computational results, the polytope $\tcf_{6}$  possesses 28895 vertices in 88 permutation orbits, whose representatives are listed in Table~\ref{table:TCFsixvertices} in the Appendix~\ref{sect:tables}. 
For the first time, also $\{0,\OT,\TT\}$-valued vertices occur, more precisely,
\begin{itemize}
\item 203 vertices in 11 orbits are $\{0,1\}$-valued, 
\item 4662 vertices in 16 orbits are $\{0,\OH\}$-valued,
\item 2430 vertices in 11 orbits are $\{0,\OH,1\}$-valued,
\item 21600 vertices in 50 orbits are $\{0,\OT,\TT\}$-valued.
\end{itemize}
%
\noindent It was not feasible to use the simple embedding of $\tcf_6$ into $\cut_6$
from Section~\ref{sect:corcut} and  \texttt{polymake} to compute the vertices by common standard hardware in reasonable time. Instead, we used  the projection of the $\Theta_{6}$ polytope in (\ref{eqn:psin}) to obtain a $\mathcal{V}$-representation for $\tcf_{6}$, from which - with some additional efforts - we extracted the vertex representation.

\begin{description}
\item[1st step:]\emph{Computing a $\mathcal{V}$-representation of $\tcf_6$.} \\ With \texttt{R} we generated the input for \texttt{polymake} (63  inequalities with 58 coefficients each) to define the polytope $\Theta_{6}$ in $\RR^{57}$. Then \texttt{polymake}  computed the 200$\phantom{.}$214 extremal points  of $\Theta_{6}$ in less then 20 minutes by standard hardware. We projected the extremal points  of $\Theta_{6}$ onto the 15 coordinates for $\tcf_{6}$, applied the coordinatewise $2-x$-transformation, and removed duplicates. This gave us 168$\phantom{.}$894 points in $[0,1]^{15}$ with convex hull $\tcf_{6}$ (a $\mathcal{V}$-representation of $\tcf_6$). 
Their coordinate values were all fractions $\frac{a}{b}, 0 \leq a \leq b \leq 9$.
\item[2nd step:] \emph{Reduction to a vertex representation of $\tcf_{6}$.}\\ 
It was not feasible to extract the subset of extremal points directly by \texttt{polymake}. Using \texttt{R} we determined the 521 permutation orbits of these 168$\phantom{.}$894 convex hull points and chose 521 representatives. These representatives included the 11 well-known representatives for $\Ex(\tcf_{6}) \cap \{0,1\}^{15}$ (i.e., the clique partition points of the complete graph $K_{6}$, see Proposition~\ref{prop:cpp}), and the 4 liftings of the 4 representatives for $\Ex(\tcf_{5}) \cap \{0,1/2\}^{10}$ described 
 above (see Table~\ref{table:TCFfivevertices} in the Appendix~\ref{sect:tables}). This gave us a list of 15 representatives known to be extremal and 506 undecided ones.\\
The extremal ones among them were identified as follows.\\
a) First, we took the union of the full permutation orbits of the 15 known representatives, a set of 1175 points, 
and added the undecided 506 candidates. The resulting list of 1681 points was handed over to \texttt{polymake}, which computed the 1259 extremal points of their convex hull (among them the previously mentioned set of 1175 points). Any candidate from the 506-list not appearing among these 1259 extremal points is a strict convex combination of points from $\tcf_{6}$, thus not extremal. This left us with the 15 representatives known to extremal plus only $84 = 1259-1175$ undecided representatives from the previous list of 506.\\ 
b) For each of the remaining 84 undecided representatives we computed with {\tt polymake}, if there is a hyperplane positively separating this selected representative from the union of all orbits of the 15 representatives known to extremal and the 83 other undecided representatives (in each case roughly 30000 points). If so, the selected representative is extremal, otherwise not. For a proof of this statement see the following Lemma~\ref{lemma:decidingextremality}. 
In this way we found $73$ extremal representatives among the $84$ undecided ones, which led to the 15+73= 88 representatives for $\Ex(\tcf_{6})$ in Table~\ref{table:TCFsixvertices} (Appendix~\ref{sect:tables}).
\end{description}

In order to justify the last step, the following lemma is needed.

\begin{lemma}\label{lemma:decidingextremality}
Let $A \subset \RR^n$ and $B\subset \RR^n$ be two disjoint finite sets with the property that 
either $B \subset \Ex(A\cup B)$  
or $B \cap \Ex(A \cup B) = \emptyset$ (property {\upshape ($*$)} in the proof).
Let $x$ be a point from $B$. 
Then 
$x \in \Ex(A\cup \{x\})$ if and only if $x \in \Ex(A \cup B)$.
\end{lemma}

\noindent (Our application in mind is $A \subset \RR^{\binom{n}{2}}$, a union of $S_{n}$-orbits, $B \subset \RR^{\binom{n}{2}}$ another $S_{n}$-orbit. Then the above condition $(*)$ holds, since $S_{n}$ acts via invertible linear maps.)

\begin{proof}
Note that the following identities hold trivially for a finite set $A \subset \RR^n$:\\
$\Ex(A) \subset A$ ($*$1) and $\conv(A) = \conv(\Ex(A))$ ($*$2).
Hence, the assertion is a consequence of the following.\\
``$\Leftarrow$'': $x \not\in \Ex(A\cup \{x\}) \overset{(*1)}{\Rightarrow} \Ex(A\cup \{x\}) \subset A \overset{(*2)}{\Rightarrow} x \in \conv(A)$, thus $x$ is a convex combination of points from $A$ (which are different from $x$, since $x \in B$, $A \cap B = \emptyset$), thus  $x \not\in \Ex(A \cup B)$.\\
``$\Rightarrow$'': $x \not\in \Ex(A \cup B) \overset{(*)}{\Rightarrow} B \cap \Ex(A \cup B) = \emptyset \overset{(*1)}{\Rightarrow} \Ex(A \cup B) \subset A \overset{(*2)}{\Rightarrow} \conv(A\cup B) \subset \conv(A) \Rightarrow x \in \conv(A)$, as above now  $x \not\in \Ex(A\cup\{x\})$ follows.
\end{proof}

\paragraph{\textbf{\upshape The facets of $\tcf_6$}}
$\phantom{a}$\\
It turned out that $\tcf_{6}$ has 18720 facets which split into 67 permutation orbits. For an  annotated complete list see Table~\ref{table:TCFsixfacets} in the Appendix~\ref{sect:tables}. The 67 representatives for $\tcf_{6}$ are grouped into 11 classes, according to their ``ancestral cut polytope generator'' (see below). The first 6 generators led to 6 classes with 17 representatives for $\tcf_{6}$, which are all hypermetric. A list of the corresponding 17~$b$-vectors is given in Table~\ref{table:hypermetricTCFsixfacets} (Appendix~\ref{sect:tables}). The remaining 5 generators induced 50~representatives and all of them are non-hypermetric (this is easily checked using Lemma~\ref{lemma:condforhypermetricity} and Remark~\ref{rk:condforhyp} for all but the 7th inequality derived from generator 9, for this one the vectors $c_{2,4},c_{2,5},c_{2,6}$ and $c_{3,4},c_{3,5},c_{3,6}$ are independent and the same reasoning as for criterion (b) of Lemma~\ref{lemma:condforhypermetricity} works). Thus, the number of hypermetric orbits is 17 out of 67 ( $\approx 25.4\%$), with 858 hypermetric facets out of 18720 (just $\approx 4.6\%$). \\
\\
We obtained this list of representatives for the facets of $\tcf_6$ in using known results about the cut polytope $\cut_{7}$ (Section~\ref{sect:corcut}), the previously computed vertex set $\Ex(\tcf_{6})$ and the software \texttt{R}:
\begin{description}
\item[1st step:]  
Choose one of the 11 homogeneous generators $g_{i} \leq 0$, $i \in \{1,\ldots,11\}$ for the  facets of the cut polytope $\cut_{7}$ (see Table~\ref{table:cutgenerators} (Appendix~\ref{sect:tables}) and Section~\ref{sect:corcut}). Compute the list of all  facets of $\cut_{7}$ generated by $g_{i}$ w.r.t.\ switchings and permutations (cf. Section~\ref{sect:corcut}). This results in an $a_{i} \times 22$ matrix with $a_{i} \leq 40320$ for all $i$ (see also \cite{dezalaurent_97} Figure~30.6.1).
\item[2nd step:]  
Apply the simple map from Proposition~\ref{prop:cutcorfacets} to all rows of the matrix from step 1. This yields a set of valid inequalities for $\tcf_{6}$ (an $a_{i} \times 16$ matrix), which is permutation invariant by construction.
Choose representatives of the permutation orbits (the largest count was 93 representatives).
\item[3rd step:] 
Use the $28\phantom{.}895$ precomputed vertices in $\Ex(\tcf_{6})$ to decide for each representative from step 2, if it defines a facet of $\tcf_{6}$. For that, first determine which vertices from $\Ex(\tcf_{6})$  solve the inequality as an equality. Then check if the rank of the matrix of solutions with an added 1-column in front is at least 15. We used the vertex set $6 \cdot \Ex(\tcf_{6})$ to make all computations integer valued, so the rank-checking procedure should be computationally reliable in this case. This gives a list of representatives for certain permutation orbits of $\tcf_{6}$-facets ``stemming from the cut polytope generator $g_{i}$''.
\item[4th step:] If done for all 11 generators, the union of the 11 lists obtained in step 3 gives a complete list of representatives of the facets of $\tcf_{6}$. This holds true, since the set of all valid inequalities obtained in the second step for all $1 \leq i \leq 11$  defines $\tcf_{6}$ by Proposition~\ref{prop:cutcorfacets}, thus we know that the facets of $\tcf_{6}$ are a subset.
Finally, we checked that representatives from different lists have different permutation orbits. Thus, the 11 lists partition a minimal set of facet representatives for $\tcf_{6}$ according to the unique ``ancestral cut polytope generator''.
\end{description}
\begin{remark}
It is feasible to generate all 116 764 facets of $\cut_{7}$ in step 1, and go through steps 2 and 3 (testing 391 representatives from step 2), to just obtain the 67 facet representatives, but then relating them to the different cut polytope generators needs extra bookkeeping.
\end{remark}
\begin{remark}
One can exploit the interaction of the permutation group actions on $\cut_{n+1}$ and $\tcf_{n}$ to avoid the large row counts in step 1 and 2. Starting from a list $h_{j} \leq c_{j}, j \in J$ of facet representatives for the cut polytope $\cut_{n+1}$ w.r.t. \emph{permutations} ($|J| =108$ in the case $n=6$) there is a way to immediately compute a list of at most $(n+1)\cdot |J|$ valid inequalities for $\tcf_{n}$ that contains a complete collection of facet representatives for $\tcf_{n}$ as a sublist (details omitted). This might get interesting if one wants to investigate $\tcf_{n}$ for $n \geq 7$ using knowledge about $\cut_{n+1}$.
\end{remark}

\section{Some open questions on the geometry of $\tcf_n$}\label{sect:openquestions}
Finally, we pursue some questions which arose while studying the convex polytope $\tcf_n$ that remained open to us. To this end, let $\psd_n \subset \RR^{E_n}$ be the space of symmetric and positive semi-definite $n\times n$ matrices in the sense of \eqref{eq:psd}. As mentioned in the introduction, it is well-known that all elements of  $\tcf_n$ are positive semi-definite, that is 
\begin{align*}
\tcf_n \subset \psd_n.
\end{align*} 
It is natural to ask whether certain subsets of inequalities from facets of $\tcf_n$ imply already positive semi-definiteness.  A simple candidate for such a question could be all facets at the exposed vertices $v_0=(0,0,\ldots,0)$ and $v_1=(1,1,\ldots,1)$ of $\tcf_n$. Let us denote the polytope which is defined by these facets by $\tcf_n(v_0,v_1)$. The following problem can be seen in a similar vein to Matheron's conjecture \citep{matheron_93}. 


{\it
\begin{enumerate}[(F)]
\item For which values of $n$ does $\tcf_n(v_0,v_1) \subset \psd_n$ hold?
\end{enumerate}
}

\noindent Therefore, let us take a closer look at the facets of $\tcf_n$ at the exposed vertices $v_0$ and $v_1$. The facets at $v_{0}=(0,0,\ldots,0)$ are just the positivity inequalities $\chi_{ij}\geq 0$, which are hypermetric with $b = \mathbbm{1}_{\{i,j\}}$.  To investigate the facets of $\tcf_n$ at $v_{1} = (1,1,\ldots,1)$, the following simple lemma is helpful.
\begin{lemma} \label{lemma:hypequalityat1} A hypermetric inequality given by $b \in \ZZ^n$ is satisfied as an equality by $v_{1}$ if and only if $\sum_{i=1}^n b_{i} \in \{0,1\}$.
\end{lemma}
\begin{proof}
$\sum_{1\leq i,j\leq n}b_{i}b_{j}\cdot 1=\sum_{i=1}^{n}b_{i}$  if and only if $(\sum_{i=1}^n {b_{i}})^2 = \sum_{i=1}^n {b_{i}}$.
\end{proof}

\noindent 
A hypermetric inequality is \emph{pure hypermetric} if its corresponding $b$-vector satisfies $b \in \{-1,0,1\}^n$.
Using this lemma and inspecting Tables~\ref{table:hypermetricTCFfacets}, \ref{table:TCFsixfacets} and \ref{table:hypermetricTCFsixfacets} we derive  the following proposition.

\begin{proposition}[Facets of $\tcf_n$ at $v_1$] \label{prop:facets_at1} $\phantom{a}$\\[-6mm]
\begin{enumerate}[a)]
\item For $n=2$ the (exceptional) facet at $v_{1}$ is pure hypermetric\\ with  $\sum_{i=1}^n b_{i} = 0$.
\item For $3 \leq n \leq 5$ the facets at $v_{1}$ are pure hypermetric with $\sum_{i=1}^n b_{i} = 1$.
\item For $ n = 6$ the facets at $v_{1}$ are hypermetric with $\sum_{i=1}^n b_{i} = 1$.\\ Some are not pure.
\end{enumerate}
\end{proposition}

\noindent 
Thus, the pure hypermetricity of the hypermetric facets at $v_{1}$ is another low-dimensional phenomenon: For $n=6$ there exist non-pure hypermetric facets at $v_1$. By the lifting property for $n \geq 3$ (Proposition~\ref{prop:zeroliftingfacets}), the same holds true for $n\geq 7$. 
However, we may ask (\cf also Lemma~\ref{lemma:hypequalityat1}):
{\it
\begin{enumerate}[(G)]
\item[(G)] Are there non-hypermetric facets at $v_1$ for $n \geq 7\,$?\\
Are there hypermetric facets at $v_1$ with $\sum_{i=1}^n b_{i} = 0$ for $n \geq 7\,$? 
\end{enumerate}
}

\noindent Let $\tcf^\hy_n(v_0,v_1)$, resp.  $\tcf^\pu_n(v_0,v_1)$, be given by the hypermetric, resp. pure hypermetric, facets at $v_{0}$ and $v_{1}$. 

\begin{remark}\label{remark:polytopv0v1}
By definition 
$\tcf_n(v_0,v_1) \subset \tcf^\hy_n(v_0,v_1) \subset \tcf^\pu_n(v_0,v_1).$ 
All these sets are polytopes, since
positivity (= the facets of $\tcf_n$ at $v_0$) and triangle inequalities (which are certainly among the pure hypermetric facets of $\tcf_n$ at $v_1$ for $n\geq 3$) suffice already to imply  $\tcf^\pu_n(v_0,v_1) \subset [0,1]^{E_n}$ (which also holds true for $n=2$), i.e., all of these sets are bounded and thus indeed polytopes. These polytopes are called {\it spindles}, since each facet contains one of the two vertices $v_{0},v_{1}$.
\end{remark}

\noindent The following proposition collects some partial answers to Question (F).

\begin{proposition}[Partial answers to Question (F)] 
\label{prop:tcfpsd}$\phantom{a}$\\[-6mm]
\begin{enumerate}[a)]
\item For $2 \leq n \leq 5$ we have\\ 
$\tcf_n(v_0,v_1)=\tcf^\hy_n(v_0,v_1)=\tcf^\pu_n(v_0,v_1) \subset \psd_n$.
\item For $n=6$ we have $\tcf_6(v_0,v_1) = \tcf^\hy_6(v_0,v_1) \neq \tcf^\pu_6(v_0,v_1)$. 
\item For $n \geq 6$ we have $\tcf^\hy_n(v_0,v_1) \not \subset \psd_n$.\\ 
In particular $\tcf_6(v_0,v_1) \not\subset \psd_6$. 
\end{enumerate}
\end{proposition}

\begin{proof}
\begin{enumerate}[a)]
\item 
The equalities follow from  Table~\ref{table:hypermetricTCFfacets}.  The inclusion  $\tcf^\pu_n(v_0,v_1) \subset \psd_n$ has been solved by hand in \cite{strokorb_13} Proposition~3.6.5. for the cases $n \leq 4$. The idea for $n=4$ was to compute the extremal points of the polytope defined by positivity and triangle inequalities and to check p.s.d.\ for them. This suffices since $\psd_{n}$ is convex. For $n=5$ we used \texttt{polymake} to compute the extremal points of the polytope defined by positivity, triangle and pentagonal inequalities (see Table \ref{table:hypermetricTCFfacets}),  and \texttt{R}  to check p.s.d.
\item This follows from Proposition \ref{prop:facets_at1}.
\item  For $n\geq6$ consider the point $x \in \RR^{E_n}$ with $x_{in}=0.5$, $1 \leq i \leq n-1$; $x_{ij}=0$ otherwise. Let $X$ denote the associated matrix. 
For $b=(b_{1},\ldots,b_{n}) \in \ZZ^n$ and $s := \sum _{i=1}^n b_{i}$ we have 
\begin{align*}
bXb^t = \sum_{i=1}^n b_{i}^2 + b_{n}\sum_{i=1}^{n-1}b_{i} = \sum_{i=1}^{n-1}b_{i}^{2} + b_{n}\cdot s.
\end{align*}
This shows $bXb^t \geq s$ for $s\in\{0,1\}$, for $s=1$ use $\sum_{i=1}^{n-1}b_{i}^{2} + b_{n} \geq \sum_{i=1}^{n-1}b_{i} + b_{n} = 1$. Thus, {all} hypermetric inequalities with $\sum_{i=1}^n b_{i} \in \{0,1\}$ are satisfied, in particular those at $v_{1}$ (c.f. Lemma \ref{lemma:hypequalityat1}), and $x$ is non-negative.  On the other hand, $x$ is not positive semi-definite: For $a=(1,\ldots,1,-2)$ the above formula shows $aXa^t = (n-1) - 2(n-3)=5-n \leq -1$.
\end{enumerate}
\end{proof}
Thus, we expect ``if and only if $n \leq 5$'' to be the answer to Question F.\\

\noindent Our final question is motivated by the following observation. 
Let $\hyp_{n}$ denoted the set of points $x=(x_{ij})_{1\leq i < j \leq n} \in \RR^{E_n}$ that satisfy all hypermetric inequalities (cf.\ Section~\ref{sect:basicfacts}).
\begin{lemma} \label{lemma:tcfhyp} 
For all $n \geq 2$ the inclusions $\tcf_n \subset \hyp_n \subset \psd_n$ hold.
\end{lemma}

\begin{proof}
The first inclusion is a reformulation of Lemma \ref{lemma:hypervalid}. Now let $x \in \hyp_n$. By assumption we have  $\sum_{1\leq i,j \leq n}b_{i}b_{j}x_{ij} \geq \sum_{i=1}^n b_{i}$  for all $b =(b_{1},\ldots,b_{n}) \in \ZZ^n$. This holds for $b$ and $-b$. Thus, $\sum_{1\leq i,j \leq n}b_{i}b_{j}x_{i,j} \geq 0$ for all $b  \in \ZZ^n$. Division by integers extends this to $\QQ^n$, and continuity to $\RR^n$.
\end{proof}

\begin{remark}
By Proposition \ref{prop:nleq5hypermetric} the sets $\tcf_{n}$ and $\hyp_{n}$ even coincide for $n \leq 5$. This is no longer true for $n\geq 6$. A point $x \in \hyp_{6} \setminus \tcf_{6}$ is given by $x_{i,6} = 1/2, 1 \leq i \leq 5$, $x_{1,2}=x_{2,3}=x_{3,4}=x_{4,5}=x_{1,5}= 1/2$, and $x_{i,j} = 0$ otherwise. Indeed, since $(\sum_{i=1}^{5} x_{i,6}) - x_{1,3} - x_{3,5}-x_{2,5}-x_{2,4}-x_{1,4} = 2.5 > 2$, the point $x$ does not satisfy a permutation of the $\tcf_{6}$-facet from Proposition~\ref{prop:nonhypfacets}. We omit the computations showing  $x \in \hyp_{6}$. By lifting, this extends to examples $x^0 \in \hyp_{n} \setminus \tcf_{n}$ for all $n\geq6$.
\end{remark}

{\it
\begin{enumerate}[(H)]
\item Do the inequalities of all hypermetric facets of $\tcf_n$  define a polytope, say $\tcf_{n}^{\hy}$, already contained in $\psd_n$?
\end{enumerate}
}

\noindent This holds true for $n \leq 5$ by Proposition \ref{prop:tcfpsd}, and remains open for $n \geq 6$. Note that $\tcf_{n}^{\hy}$ is a polytope by $\tcf_{n}^{hyp} \subset \tcf_{n}^{hyp}(v_{0},v_{1})$ and Remark \ref{remark:polytopv0v1}.


\section*{Discussion}\label{sect:discussion}

In this article, we deal with the realization problem for the tail correlation function (TCF), which is an omnipresent bivariate tail dependence measure in the extremes literature. 
We make this specific by formulating Questions~(A)-(E) in the introduction.
Here, we discuss our contribution to these questions.
In doing so we address Questions~(A)-(E) partially in reversed order according to their growing complexity.

Questions~(E) and (D) can be answered fully and affirmatively by  
Corollary~\ref{cor:basicoperations} and Theorem~\ref{thm:TCFisMAX}, respectively. That is, convex combinations, products and pointwise limits are admissible operations on the set of TCFs and Theorem~\ref{thm:TCFisMAX} shows that the class of TM processes, a subclass of max-stable processes, is rich enough to realize any given TCF. Concerning the regularity of the corresponding TM process, we identify continuity of its TCF as a necessary and sufficient condition for its stochastic continuity (Corollary~\ref{cor:chicty}), which contributes to Question~(C). Theorem~\ref{thm:TCFisMAX} also opens up links to binary ($\{0,1\}$-valued) processes and thereby provides a substantial reduction of Questions~(A) and (B). Corollary~\ref{cor:TCFextension} reduces them even further to the study of TCFs on finite base spaces. Together with Corollary~\ref{cor:finitepolytope}, we reveal that membership in the set of TCFs (even on infinite spaces) can be completely characterized by a system of affine inequalities, which -- if known -- would provide a complete answer to Question~(A). 

To identify and classify these affine inequalities, a better understanding of the geometry of the polytope $\tcf_n$ of $n\times n$ tail correlation functions (matrices) for arbitrary $n$ is needed. Its facet inducing inequalities constitute such a list (actually, an $\mathcal{H}$-representation would suffice already).  
Lemma~\ref{lemma:hypervalid} contributes to Question~(A) in that it provides a rich class of necessary conditions (all hypermetric inequalities) for membership in $\tcf_n$, whereas Proposition~\ref{prop:cpp} identifies any clique partition point to be an admissible TCF. In Section~\ref{sect:corcut}, we discuss that the polytope $\tcf_n$ can be viewed either as an affine \emph{projection} of the polytope $\Theta_n$ (whose facets are well-understood) or as an affine \emph{intersection} with the \emph{correlation polytope} (whose vertices are well-understood). Both views immediately suggest algorithms that can be easily implemented in order to obtain the vertices and facets of $\tcf_n$ that in theory would work ``for arbitrary $n$''. This would solve Question~(A) computationally.  
Due to the complexity of the problem, software computations lead to a full description of facets and vertices of $\tcf_n$ only up to $n=6$ (Section~\ref{sect:compresults}).

Indeed, several of our results reveal the rapidly growing complexity of Question~(A) as $n$ grows. 
Starting from $n=3$, no facet inducing inequality of $\tcf_n$ will ever become obsolete (Proposition~\ref{prop:zeroliftingfacets}). 
For instance, the triangle inequality $(\ref{eqn:triangle})$ cannot be deduced from any other set of valid inequalities for $\tcf_n$.
By contrast, \emph{all} facet inducing inequalities that define the polytope of ECFs $\Theta_{n}$ become obsolete for $\Theta_{n'}$ for higher $n'>n$, 
and still $\Theta_n$ has $2^n$ facets in dimension $n$.
Starting from $n \geq 6$ there exist (actually plenty of) non-hypermetric facets of $\tcf_{n}$ (Proposition~\ref{prop:nonhypfacets}).
Moreover, we derived the facets of $\tcf_6$ from the facets of the cut polytope $\cut_7$ which had 11 generators for 116$\phantom{.}$764 facets. The next step would take into account the polytope $\cut_{8}$, which has already more than 217 million facets which can be subdivided into 147 orbits under permutations and switchings \citep[p.~505]{dezalaurent_97}.  
It is even possible to choose $n$ sufficiently large, such that a given finite set of rational numbers from the interval $[0,1]$ turns up as coordinate values of a single vertex of $\tcf_n$ (Proposition~\ref{prop:unbounded}).
Altogether, these results confound the aim of a full answer to Question~(A).

Finally, if Question~(A) is already so difficult to answer, what more can be eventually said about Question~(B)? That is, given a TCF $\chi$, say on a finite space, how to construct a specific stochastic model that realizes $\chi$? Again, from our dual views on $\tcf_n$ as affine ``projection of''  or ``intersection with'' other polytopes, it is easy to formulate naive ad-hoc algorithms providing an entire convex polytope of solutions to such a problem, cf.\ \cite{strokorb_13}, p.~65. Perhaps more interestingly, in case of  $T=\RR^d$, \cite{strokorbballanischlather_15} characterize subclasses of radially symmetric and monotonously decreasing TCFs with some sharp bounds on membership in the class of TCFs on $\RR^d$ (cf.\ Table 2 therein) and recover realizing max-stable models. Surprisingly often, it is possible to obtain explicitly several such realizing models sharing an identical TCF, but with rather different spectral profiles. 
In this sense, the reader should not overrate the finding of a specific model meeting a given TCF even though the TM models helped us here to approach the realization problem.

To conclude with, 
independently of our research \citep{fss_14} and motivated from an insurance context, \cite{ehw15} dealt with almost the same questions (in particular Questions A and B) for random vectors with an emphasis on the construction of realizing copulas as we learned on the EVA 2015 in AnnArbor. 
Our approach offers (at least theoretically) an algorithm that can solve Questions A and B for random vectors completely (even though the feasibilty of such an algorithm breaks down very quickly as the dimension grows and we have doubts on its practical use in higher dimensions). This answers one of the questions raised in the discussion of \cite{ehw15}.

\paragraph{Acknowledgements}
The authors would like to thank two referees and an AE who provided many thoughtful comments leading to substantial improvements in the presentation of this material. We are also thankful to be made aware of the regularity question and the  early works of Shepp on unit covariances.
K.~Strokorb undertook part of this work as part of her PhD thesis \cite{strokorb_13} as a member of the Research Training Group 1023 and gratefully acknowledges financial support by the German Research Foundation DFG. 

{\small
\bibliographystyle{spbasic}      
\bibliography{lit.bib}           
}


\newpage
\appendix
\section{Tables}\label{sect:tables}

\begin{table}[h]\footnotesize \centering
\setlength{\tabcolsep}{1.1mm}
\begin{tabular}{lp{2mm}rrrrrp{2mm}rrrrr} 
\toprule
\multicolumn{13}{c}{\textbf{Vertex and facet counts}}\\
\midrule
&& \multicolumn{5}{c}{$\tcf_n$} && \multicolumn{5}{c}{$\Theta_n$} \\
\cmidrule{3-7} \cmidrule{9-13} 
$n$ & &2&3&4&5&6 & &2&3&4&5&6 \\
\cmidrule{1-1} \cmidrule{3-7} \cmidrule{9-13} 
$\# $ vertices 
& &2&5&15&214&28$\phantom{.}$895 
& &2&6&42&1292&200$\phantom{.}$214 \\
$\# $ facets 
& &2&6&22&110&18$\phantom{.}$720 
& &2&7&15&31&63 \\ 
$\# $ permutation orbits of vertices 
& &2&3&5&11&88 
& &2&4&10&45&583 \\ 
$\# $ permutation orbits of facets 
& &2&2&3&7&67 
& &2&3&4&5&6 \\ 
\midrule
&& \multicolumn{5}{c}{$\cor_n$} && \multicolumn{5}{c}{$\cut_{n+1}$} \\
\cmidrule{3-7} \cmidrule{9-13} 
$n$ & &2&3&4&5&6 & &2&3&4&5&6  \\
\cmidrule{1-1} \cmidrule{3-7} \cmidrule{9-13} 
$\# $ vertices 
&  &4&8&16&32&64 
&  &4&8&16&32&64\\ 
$\# $ facets 
& &4&16&56&368&116$\phantom{.}$764 
& &4&16&56&368&116$\phantom{.}$764 \\ 
$\# $ permutation orbits of vertices 
& &3&4&5&6&7 
& &2&3&3&4&4\\ 
$\# $ permutation orbits of facets 
& &3&5&10&29&428 
& &2&2&5&11&108 \\ 
$\# $ perm./switch. orbits of vertices 
&&&&&& & &1&1&1&1&1\\ 
$\# $ permu./switch. orbits of facets 
&&&&&& & &1&1&2&3&11 \\ 
\bottomrule
\end{tabular}
\caption{Vertex and facet counts for the polytope of tail correlation functions $\tcf_{n} \subset \RR^{\binom{n}{2}}$, the polytope of extremal coefficient functions $\Theta_n \subset \RR^{2^n-n-1}$, the correlation polytope $\cor_n \subset \RR^{n+\binom{n}{2}}$ and the cut polytope $\cut_n \subset \RR^{\binom{n+1}{2}}$. 
For $\Theta_n$ the number of facets ($2^n-1$) and orbits of facets ($n$) follow from Lemma \ref{lemma:CA_finite}. Since $\cor_n$ and $\cut_{n+1}$ are linearly equivalent, they have the same number of vertices ($2^n$ by definition) and facets (see \cite{dezalaurent_97} p.~503-505 for the respective numbers as well as for the permutation/switching orbits of $\cut_{n+1}$). All other numbers rely on computations using the software \texttt{R} and \texttt{polymake}. The counts for $\tcf_n$ and $\Theta_n$ in case $n\leq 4$ have been obtained previously ``by hand'' in \cite{strokorb_13} p.~62-63.
\label{table:counts}
}

\end{table}

\begin{table}[p]\footnotesize \centering
\begin{tabular}{cccccccccc|c|l}
\toprule
\multicolumn{12}{c}{\textbf{Vertices of $\tcf_5$}}  \\
\midrule
\multicolumn{12}{l}{7 $\{0,1\}$-valued representatives}  \\
\midrule
0 & 0 & 0 & 0 & 0 & 0 & 0 & 0 & 0 & 0 & 1 & five 1-cliques\\
0 & 0 & 0 & 0 & 0 & 0 & 0 & 0 & 0 & 1 & 10 & one 2-clique\\
0 & 0 & 0 & 0 & 0 & 0 & 0 & 1 & 1 & 1 & 10 & one 3-clique\\
0 & 0 & 0 & 0 & 0 & 0 & 1 & 1 & 0 & 0 & 15 & two 2-cliques\\
0 & 0 & 0 & 0 & 1 & 1 & 1 & 1 & 1 & 1 & 5 & one 4-clique \\
0 & 0 & 0 & 1 & 1 & 1 & 0 & 1 & 0 & 0 & 10 & one 2-clique and one 3-clique\\
1 & 1 & 1 & 1 & 1 & 1 & 1 & 1 & 1 & 1 & 1 & one 5-clique \\
\midrule
\multicolumn{12}{l}{4 $\{0,\OH\}$-valued representatives}  \\
\midrule
0 & 0 & \OH & \OH & \OH & \OH & \OH & \OH & \OH & \OH & 30 &\\ 
0 & 0 & \OH & \OH & \OH & 0 & \OH & \OH & \OH & \OH & 60 &\\ 
0 & 0 & 0 & \OH & \OH & \OH & 0 & \OH & \OH & \OH & 60 &\\ 
\OH & 0 & 0 & \OH & \OH & 0 & 0 & \OH & 0 & \OH & 12 &\\ 
\bottomrule
\end{tabular}
\caption{
The 11 representatives $(\chi_{1,2},\ldots,\chi_{4,5})$ for the 214 elements of $\Ex(\tcf_{5})$.
Columns (1)-(10) list the coordinates $\chi_{1,2},\ldots,\chi_{4,5}$, Column (11) gives the orbit length under permutations and the last column is a comment on the generating clique partition (see Section~\ref{sect:basicfacts}).
\label{table:TCFfivevertices}}
\vspace{10mm}
\setlength{\tabcolsep}{2mm}
\begin{tabular}{llll|r|l}
\toprule
\multicolumn{6}{c}{\textbf{Facets of $\tcf_n$  for $2\leq n \leq 5$}}  \\
\midrule
$n=2$ & $b=(1,1)$ & \emph{positivity inequality} $x_{1,2}\geq 0$ & (\ol $\phantom{0}$1)  & 2 facets &  $v_0$\\
& $b=(1,-1)$ & $x_{1,2}\leq 1$ (this facet disappears for $n \geq 3$) & (\ol $\phantom{0}$1) && $v_1$\\
\midrule
$n=3$ & $b=(1,1,0)$ & lifting of \emph{positivity ineq.} & (\ol $\phantom{0}$3) & 6 facets & $v_0$\\
& $b=(1,1,-1)$ & \emph{triangle inequality} &  (\ol $\phantom{0}$3) && $v_1$\\
\midrule
$n=4$ & $b=(1,1,0,0)$ & lifting of \emph{positivity ineq.} & (\ol $\phantom{0}$6) & 22 facets & $v_0$ \\
& $b=(1,1,-1,0)$ & lifting of  \emph{triangle ineq.} & (\ol $12$) &&  $v_1$\\
& $b=(1,1,1,-1)$ & \emph{tetrahedron inequality} & (\ol $\phantom{0}$4) && \\
\midrule
$n=5$ & $b=(1,1,0,0,0)$ & lifting of \emph{positivity ineq.} &  (\ol $10$) & 110 facets & $v_0$\\
& $b=(1,1,-1,0,0)$ & lifting of \emph{triangle ineq.} &  (\ol $30$) && $v_1$\\
& $b=(1,1,1,-1,0)$ & lifting of \emph{tetrahedron ineq.} & (\ol $20$) &&  \\
& $b=(1,1,1,1,-1)$ & \emph{pyramid inequality} & (\ol $\phantom{0}$5) && \\
& $b=(1,1,1,1,-2)$ & \emph{2-weighted variant of pyramid ineq.} & (\ol $\phantom{0}$5) && \\
& $b=(1,1,1,-1,-1)$ & \emph{pentagonal inequality} &  (\ol $10$) && $v_1$\\
& $b=(2,1,1,-1,-1)$ & \emph{2-weighted variant of pentagonal ineq.} & (\ol $30$) && \\
\bottomrule
\end{tabular}
\caption{
Permutation orbit representatives for the facets of $\tcf_n$ for $2 \leq n \leq 5$. Since all facets are hypermetric, they can be described by their corresponding $b$-vector (see Section~\ref{sect:basicfacts}). The number in brackets is the orbit length.  The last column indicates whether the respective facet contains one of the exposed vertices $v_0=(0,0,\dots,0)$ or $v_1=(1,1,\dots,1)$.
\label{table:hypermetricTCFfacets}}
\end{table}

\def\a{0.6 mm} 
\def\b{0.0 mm}

\setlength{\tabcolsep}{0.5mm}
\begin{table}[p]\scriptsize
\begin{tabular}{ccccccccccccccc|c}
\toprule
\multicolumn{16}{l}{\textbf{\scriptsize Vertices of $\tcf_6$}}  \\[\b]
\midrule
\multicolumn{16}{l}{\scriptsize 7 $\{0,1\}$-vd.\ repr'tives (liftings from $\tcf_{5}$)}  \\[\b]
\midrule
0 & 0 & 0 & 0 & 0 & 0 & 0 & 0 & 0 & 0 & 0 & 0 & 0 & 0 & 0 & 1\\[\b]
0 & 0 & 0 & 0 & 0 & 0 & 0 & 0 & 0 & 0 & 0 & 0 & 0 & 0 & 1 & 15\\[\b]
0 & 0 & 0 & 0 & 0 & 0 & 0 & 0 & 0 & 0 & 0 & 0 & 1 & 1 & 1 & 20\\[\b]
0 & 0 & 0 & 0 & 0 & 0 & 0 & 0 & 0 & 0 & 0 & 1 & 1 & 0 & 0 & 45\\[\b]
0 & 0 & 0 & 0 & 0 & 0 & 0 & 0 & 0 & 1 & 1 & 1 & 1 & 1 & 1 & 15\\[\b]
0 & 0 & 0 & 0 & 0 & 0 & 0 & 0 & 1 & 1 & 1 & 0 & 1 & 0 & 0 & 60\\[\b]
0 & 0 & 0 & 0 & 0 & 1 & 1 & 1 & 1 & 1 & 1 & 1 & 1 & 1 & 1 & 6\\[\b]
\midrule
\multicolumn{16}{l}{\scriptsize 4 new $\{0,1\}$-vd.\ repr'tives (not liftings)}  \\[\b]
\midrule
0 & 0 & 0 & 0 & 1 & 0 & 0 & 1 & 0 & 1 & 0 & 0 & 0 & 0 & 0 & 15\\[\b]				
0 & 0 & 0 & 0 & 1 & 1 & 1 & 1 & 0 & 1 & 1 & 0 & 1 & 0 & 0 & 15\\[\b]
0 & 0 & 0 & 1 & 1 & 1 & 1 & 0 & 0 & 1 & 0 & 0 & 0 & 0 & 1 & 10\\[\b]
1 & 1 & 1 & 1 & 1 & 1 & 1 & 1 & 1 & 1 & 1 & 1 & 1 & 1 & 1 & 1\\[\b]
\midrule
\multicolumn{16}{l}{4 $\{0,\OH\}$-vd.\ repr'tives (liftings from $\tcf_{5}$)}  \\[\b]
\midrule
0 & 0 & 0 & 0 & 0 & 0 & 0 & 0 & \OH & \OH & \OH & 0 & \OH & \OH & \OH & 360\\[\a]
0 & 0 & 0 & 0 & 0 & 0 & 0 & \OH & \OH & \OH & 0 & \OH & \OH & 0 & 0 & 72\\[\a]
0 & 0 & 0 & 0 & 0 & 0 & 0 & \OH & \OH & \OH & 0 & \OH & \OH & \OH & \OH & 360\\[\a]
0 & 0 & 0 & 0 & 0 & 0 & 0 & \OH & \OH & \OH & \OH & \OH & \OH & \OH & \OH & 180\\[\a]
\midrule
\multicolumn{16}{l}{\scriptsize 12 new $\{0,\OH\}$-vd.\ repr'tives (not liftings)}  \\[\b]
\midrule
0 & 0 & 0 & 0 & \OH & 0 & \OH & \OH & 0 & \OH & \OH & \OH & \OH & 0 & 0 & 360\\[\a]
0 & 0 & 0 & 0 & \OH & 0 & \OH & \OH & 0 & \OH & \OH & \OH & \OH & 0 & \OH & 720\\[\a]
0 & 0 & 0 & 0 & \OH & 0 & \OH & \OH & 0 & \OH & \OH & \OH & \OH & \OH & \OH & 360\\[\a]
0 & 0 & 0 & \OH & \OH & 0 & \OH & 0 & \OH & \OH & \OH & 0 & 0 & 0 & \OH & 360\\[\a]
0 & 0 & 0 & \OH & \OH & 0 & \OH & 0 & \OH & \OH & \OH & 0 & \OH & \OH & \OH & 120\\[\a]
0 & 0 & 0 & \OH & \OH & \OH & \OH & 0 & 0 & \OH & 0 & 0 & \OH & \OH & \OH & 180\\[\a]
0 & 0 & 0 & \OH & \OH & \OH & \OH & 0 & 0 & \OH & \OH & \OH & \OH & \OH & \OH & 90\\[\a]
0 & 0 & 0 & \OH & \OH & \OH & \OH & 0 & \OH & \OH & 0 & \OH & \OH & \OH & \OH & 360\\[\a]
0 & 0 & 0 & \OH & \OH & \OH & \OH & 0 & \OH & \OH & \OH & 0 & \OH & \OH & 0 & 360\\[\a]
0 & 0 & 0 & \OH & \OH & \OH & \OH & 0 & \OH & \OH & \OH & 0 & \OH & \OH & \OH & 360\\[\a]
0 & 0 & 0 & \OH & \OH & \OH & \OH & \OH & \OH & \OH & \OH & \OH & \OH & \OH & \OH & 60\\[\a]
0 & 0 & \OH & \OH & \OH & \OH & 0 & \OH & \OH & \OH & 0 & \OH & \OH & \OH & \OH & 360\\[\a]	
\midrule
\multicolumn{16}{l}{\scriptsize 11 new $\{0,\OH,1\}$-vd.\ repr'tives (not liftings)}  \\[\b]
\midrule
0 & 0 & 0 & 0 & \OH & 1 & \OH & \OH & 0 & \OH & \OH & 0 & \OH & \OH & \OH & 180\\[\a]	
0 & 0 & 0 & 0 & \OH & 1 & \OH & \OH & \OH & \OH & \OH & \OH & \OH & 0 & \OH & 360\\[\a]
0 & 0 & 0 & 1 & \OH & \OH & \OH & 0 & 0 & \OH & 0 & \OH & 0 & \OH & \OH & 180\\[\a]
0 & 0 & 0 & \OH & \OH & 1 & \OH & 0 & \OH & \OH & 0 & \OH & \OH & 0 & 0 & 180\\[\a]
0 & 0 & 0 & \OH & \OH & 1 & \OH & 0 & \OH & \OH & 0 & \OH & \OH & \OH & \OH & 360\\[\a]
0 & 0 & 0 & \OH & \OH & 1 & \OH & \OH & \OH & \OH & \OH & \OH & 0 & 0 & \OH & 180\\[\a]
0 & 0 & 0 & \OH & \OH & 1 & \OH & \OH & \OH & \OH & \OH & \OH & 0 & \OH & \OH & 360\\[\a]
0 & 0 & 0 & \OH & \OH & 1 & \OH & \OH & \OH & \OH & \OH & \OH & \OH & \OH & \OH & 180\\[\a]
0 & 0 & 1 & \OH & \OH & \OH & 0 & \OH & \OH & 0 & \OH & \OH & \OH & \OH & \OH & 90\\[\a]
0 & 0 & \OH & \OH & \OH & \OH & 0 & \OH & \OH & \OH & \OH & \OH & \OH & \OH & 1 & 180\\[\a]
0 & 0 & \OH & \OH & \OH & \OH & \OH & \OH & \OH & \OH & \OH & \OH & 1 & \OH & \OH & 180\\[\a]
\midrule
\multicolumn{16}{l}{\scriptsize 50 new $\{0,\OT,\TT\}$-vd.\ repr'tives (not liftings)}  \\[\b]
\midrule
0 & 0 & 0 & 0 & \OT & \OT & \OT & \OT & 0 & \TT & \TT & \TT & \TT & \TT & \TT & 120\\[\a]
\bottomrule
\end{tabular}
\hspace{1.2mm}
\begin{tabular}{ccccccccccccccc|c}
\toprule
0 & 0 & 0 & 0 & \OT & \OT & \OT & \OT & \TT & \TT & \TT & 0 & \TT & \OT & \OT & 360\\[\a]
0 & 0 & 0 & 0 & \OT & \TT & \TT & \TT & \OT & \TT & \TT & \OT & \TT & \TT & \TT & 180\\[\a]
0 & 0 & 0 & 0 & \TT & \OT & \OT & \OT & \OT & \TT & \TT & 0 & \TT & 0 & \OT & 360\\[\a]
0 & 0 & 0 & 0 & \TT & \TT & \TT & \TT & 0 & \TT & \TT & \OT & \TT & \OT & \OT & 120\\[\a]
0 & 0 & 0 & \OT & \OT & \OT & \OT & 0 & \TT & \TT & \TT & \OT & \TT & \OT & \OT & 360\\[\a]
0 & 0 & 0 & \OT & \OT & \OT & \OT & \TT & \TT & \TT & 0 & \OT & \OT & \OT & \TT & 720\\[\a]
0 & 0 & 0 & \OT & \OT & \TT & \TT & 0 & \TT & \TT & \OT & \TT & \OT & \TT & \OT & 360\\[\a]
0 & 0 & 0 & \OT & \OT & \TT & \TT & \OT & \OT & \TT & \OT & \TT & \TT & \TT & \TT & 720\\[\a]
0 & 0 & 0 & \OT & \OT & \TT & \TT & \OT & \TT & \TT & \OT & \TT & \TT & \OT & \OT & 360\\[\a]
0 & 0 & 0 & \OT & \TT & \OT & \OT & 0 & \OT & \TT & \TT & 0 & \TT & 0 & 0 & 360\\[\a]
0 & 0 & 0 & \OT & \TT & \OT & \OT & \TT & \OT & \TT & 0 & 0 & \OT & 0 & \OT & 360\\[\a]
0 & 0 & 0 & \OT & \TT & \OT & \OT & \TT & \OT & \TT & \OT & 0 & \OT & \OT & \OT & 720\\[\a]
0 & 0 & 0 & \OT & \TT & \TT & \TT & 0 & 0 & \TT & \OT & 0 & \OT & \OT & \OT & 720\\[\a]
0 & 0 & 0 & \OT & \TT & \TT & \TT & \OT & 0 & \TT & \OT & \OT & \TT & \OT & \OT & 720\\[\a]
0 & 0 & 0 & \OT & \TT & \TT & \TT & \OT & \OT & \TT & \OT & \OT & \TT & 0 & 0 & 360\\[\a]
0 & 0 & 0 & \OT & \TT & \TT & \TT & \OT & \OT & \TT & \TT & \OT & \TT & \OT & \OT & 360\\[\a]
0 & 0 & 0 & \TT & \TT & \TT & \TT & 0 & \OT & \TT & \OT & \OT & \OT & \OT & \TT & 360\\[\a]
0 & 0 & \OT & \OT & \OT & \OT & 0 & \OT & \OT & \TT & \TT & \TT & \TT & \TT & \TT & 180\\[\a]
0 & 0 & \OT & \OT & \OT & \OT & 0 & \TT & \TT & \TT & \OT & \OT & \OT & \OT & \TT & 360\\[\a]
0 & 0 & \OT & \OT & \OT & \TT & 0 & \OT & \TT & \OT & \OT & \TT & \TT & \OT & \OT & 720\\[\a]
0 & 0 & \OT & \OT & \OT & \TT & \OT & \OT & \OT & \OT & \TT & \TT & \TT & \TT & \TT & 360\\[\a]
0 & 0 & \OT & \OT & \OT & \TT & \OT & \OT & \TT & \OT & \OT & \TT & \TT & 0 & 0 & 180\\[\a]
0 & 0 & \OT & \OT & \OT & \TT & \OT & \OT & \TT & \OT & \OT & \TT & \TT & 0 & \OT & 360\\[\a]
0 & 0 & \OT & \OT & \OT & \TT & \OT & \OT & \TT & \OT & \TT & \OT & \TT & \TT & \OT & 360\\[\a]
0 & 0 & \OT & \OT & \OT & \TT & \OT & \OT & \TT & \OT & \TT & \TT & \TT & \OT & \TT & 720\\[\a]
0 & 0 & \OT & \OT & \OT & \TT & \OT & \OT & \TT & \TT & \TT & \TT & \TT & \TT & \TT & 360\\[\a]
0 & 0 & \OT & \OT & \OT & \TT & \OT & \TT & \TT & \OT & \TT & \TT & 0 & \OT & \TT & 360\\[\a]
0 & 0 & \OT & \OT & \OT & \TT & \OT & \TT & \TT & \TT & \OT & \TT & \OT & \TT & \TT & 360\\[\a]
0 & 0 & \OT & \OT & \OT & \TT & \TT & \TT & \TT & \TT & \TT & \TT & \TT & \TT & \TT & 60\\[\a]
0 & 0 & \OT & \OT & \TT & \OT & 0 & \TT & \OT & \TT & 0 & \OT & 0 & \OT & \OT & 360\\[\a]
0 & 0 & \OT & \OT & \TT & \OT & \OT & \OT & \OT & \TT & \TT & \OT & \TT & \OT & \OT & 360\\[\a]
0 & 0 & \OT & \OT & \TT & \TT & \OT & \OT & 0 & \OT & \OT & \OT & \TT & \OT & \TT & 720\\[\a]
0 & 0 & \OT & \OT & \TT & \TT & \OT & \OT & 0 & \OT & \OT & \OT & \TT & \TT & \TT & 360\\[\a]
0 & 0 & \OT & \OT & \TT & \TT & \OT & \OT & \OT & \OT & \TT & 0 & \TT & \OT & 0 & 720\\[\a]
0 & 0 & \OT & \OT & \TT & \TT & \OT & \OT & \OT & \OT & \TT & \OT & \TT & 0 & \OT & 720\\[\a]
0 & 0 & \OT & \OT & \TT & \TT & \OT & \TT & 0 & \OT & \TT & \OT & 0 & \OT & \OT & 720\\[\a]
0 & 0 & \OT & \OT & \TT & \TT & \OT & \TT & 0 & \OT & \TT & \OT & \OT & \TT & \OT & 720\\[\a]
0 & 0 & \OT & \OT & \TT & \TT & \OT & \TT & \OT & \OT & \TT & \OT & \OT & 0 & \OT & 360\\[\a]
0 & 0 & \OT & \OT & \TT & \TT & \OT & \TT & \OT & \TT & \TT & \OT & \TT & \OT & \OT & 720\\[\a]
0 & 0 & \OT & \OT & \TT & \TT & \TT & \TT & 0 & \TT & \TT & \OT & \TT & \OT & \OT & 360\\[\a]
0 & 0 & \OT & \TT & \TT & \TT & \OT & \OT & \OT & \TT & 0 & \OT & 0 & \OT & \TT & 360\\[\a]
0 & 0 & \OT & \TT & \TT & \TT & \TT & \OT & \OT & \TT & \OT & \OT & \OT & \OT & \TT & 180\\[\a]
0 & 0 & \TT & \TT & \TT & \OT & \OT & \OT & \OT & \OT & \OT & \OT & \TT & \TT & \TT & 60\\[\a]
0 & \OT & \OT & \OT & \OT & \OT & \OT & \OT & \TT & \TT & \TT & \OT & \TT & \TT & \TT & 360\\[\a]
0 & \OT & \OT & \OT & \OT & \OT & \OT & \TT & \TT & \TT & \OT & \TT & \TT & \TT & \TT & 720\\[\a]
0 & \OT & \OT & \OT & \OT & \OT & \TT & \TT & \TT & \OT & \TT & \TT & \TT & \TT & \TT & 360\\[\a]
0 & \OT & \OT & \OT & \TT & \OT & \OT & \TT & \OT & \TT & 0 & \TT & \OT & \TT & \OT & 720\\[\a]
0 & \OT & \OT & \OT & \TT & \OT & \TT & \TT & \OT & 0 & \OT & \TT & \TT & \OT & \OT & 360\\[\a]
0 & \OT & \OT & \OT & \TT & \OT & \TT & \TT & \OT & \TT & \TT & \OT & \TT & \OT & \OT & 360\\[\a]
\bottomrule
\end{tabular}
\caption{
The 88 $\{...\}$-valued representatives (vd.\ repr'tives) $(\chi_{1,2},\ldots,\chi_{5,6})$ for the 28895 elements of $\Ex(\tcf_{6})$ (see Section~\ref{sect:compresults}). 
Columns (1)-(15) list the coordinates $\chi_{1,2},\ldots,\chi_{5,6}$. The last column gives the orbit length under permutations.
\label{table:TCFsixvertices}
}
\end{table}

\setlength{\tabcolsep}{0.68mm}
\begin{table}[p]\scriptsize
\begin{tabular}{ccccccccccccccc|c||c|c}
\toprule
\multicolumn{18}{l}{\textbf{\scriptsize Facets of $\tcf_6$}}  \\
\midrule
\multicolumn{18}{l}{\scriptsize Generator 1}  \\
\midrule
\multicolumn{14}{l}{{\scriptsize Positivity}} & & & &  \\
-1 & 0 & 0 & 0 & 0 & 0 & 0 & 0 & 0 & 0 & 0 & 0 & 0 & 0 & 0 & 0 & 7657 & 15 \\
 \multicolumn{14}{l}{{Triangle inequality}} & & & &  \\
-1 & 1 & 0 & 0 & 0 & 1 & 0 & 0 & 0 & 0 & 0 & 0 & 0 & 0 & 0 & 1 & 3521 & 60  \\ 
\midrule
\multicolumn{18}{l}{\scriptsize Generator 2}   \\
\midrule
\multicolumn{14}{l}{{\scriptsize Tetrahedron inequality}}  & & & & \\
-1 & -1 & 1 & 0 & 0 & -1 & 1 & 0 & 0 & 1 & 0 & 0 & 0 & 0 & 0 & 1 & 1554 & 60 \\
\multicolumn{14}{l}{{\scriptsize Pentagonal inequality}}  & & & & \\
-1 & -1 & 1 & 1 & 0 & -1 & 1 & 1 & 0 & 1 & 1 & 0 & -1 & 0 & 0 & 2 & 1043 &  60 \\
\midrule
\multicolumn{18}{l}{\scriptsize Generator 3}   \\
\midrule
\multicolumn{14}{l}{{\scriptsize Pyramid inequality}}  & & & & \\
-1 & -1 & -1 & 1 & 0 & -1 & -1 & 1 & 0 & -1 & 1 & 0 & 1 & 0 & 0 & 1 & 110 & 30 \\
\multicolumn{14}{l}{{\scriptsize 2-weighted pentagonal inequality}} & & & & \\
-2 & -2 & 2 & 2 & 0 & -1 & 1 & 1 & 0 & 1 & 1 & 0 & -1 & 0 & 0 & 3 & 135 & 180 \\
\multicolumn{14}{l}{{\scriptsize 2-weighted pyramid inequality}} & & & & \\
-1 & -1 & -1 & 2 & 0 & -1 & -1 & 2 & 0 & -1 & 2 & 0 & 2 & 0 & 0 & 3 & 102 & 30 \\
\multicolumn{14}{l}{{\scriptsize ``new inequalities'' from here on}} & & & & \\
-2 & -2 & 2 & 2 & 2 & -1 & 1 & 1 & 1 & 1 & 1 & 1 & -1 & -1 & -1 & 4 & 129 & 60 \\
-1 & -1 & -1 & 1 & 2 & -1 & -1 & 1 & 2 & -1 & 1 & 2 & 1 & 2 & -2 & 4 & 129 & 30 \\
\midrule
\multicolumn{18}{l}{\scriptsize Generator 4}  \\
\midrule
-1 & -1 & -1 & 1 & 1 & -1 & -1 & 1 & 1 & -1 & 1 & 1 & 1 & 1 & -1 & 2 & 554 & 15 \\
\midrule
\multicolumn{18}{l}{\scriptsize Generator 5}  \\
\midrule
-2 & -2 & -2 & 2 & 2 & -1 & -1 & 1 & 1 & -1 & 1 & 1 & 1 & 1 & -1 & 3 & 20 & 60 \\
-1 & -1 & -1 & -1 & 2 & -1 & -1 & -1 & 2 & -1 & -1 & 2 & -1 & 2 & 2 & 3 & 20 & 6 \\
-4 & -2 & 2 & 2 & 2 & -2 & 2 & 2 & 2 & 1 & 1 & 1 & -1 & -1 & -1 & 5 & 61 & 60 \\
-2 & -2 & -2 & 2 & 4 & -1 & -1 & 1 & 2 & -1 & 1 & 2 & 1 & 2 & -2 & 5 & 53 & 120 \\
\midrule
\multicolumn{18}{l}{\scriptsize Generator 6} \\
\midrule
-1 & -1 & -1 & -1 & 1 & -1 & -1 & -1 & 1 & -1 & -1 & 1 & -1 & 1 & 1 & 1 & 15 & 6 \\
-3 & -3 & 3 & 3 & 3 & -1 & 1 & 1 & 1 & 1 & 1 & 1 & -1 & -1 & -1 & 6 & 15 & 60 \\
-1 & -1 & -1 & -1 & 3 & -1 & -1 & -1 & 3 & -1 & -1 & 3 & -1 & 3 & 3 & 6 & 15 & 6  \\
\midrule
\multicolumn{18}{l}{\scriptsize Generator 7: Clique-Web-Generator;}  \\
\multicolumn{18}{l}{{\scriptsize all following inequalities are not hypermetric}} \\
\midrule
-1 & -1 & -1 & 1 & 1 & -1 & -1 & 1 & 1 & 0 & 0 & 1 & 1 & 0 & -1 & 2 & 95 & 180 \\
-1 & -1 & 0 & 0 & 1 & 0 & -1 & 0 & 1 & 0 & -1 & 1 & -1 & 1 & 1 & 2 & 95 & 72 \\
-1 & -1 & -1 & 1 & 1 & -1 & -1 & 1 & 1 & 0 & 0 & 1 & 1 & 1 & 0 & 3 & 15 & 360 \\
-1 & -1 & -1 & 1 & 1 & -1 & 0 & 0 & 1 & 0 & 1 & 0 & 1 & 1 & 0 & 3 & 15 & 360 \\
-1 & -1 & -1 & 1 & 1 & -1 & 0 & 0 & 1 & 0 & 1 & 1 & 1 & 1 & -1 & 3 & 15 & 720 \\
-1 & -1 & 0 & 0 & 1 & -1 & 1 & 1 & 1 & 1 & 1 & 1 & -1 & -1 & 0 & 3 & 15 & 360 \\
-1 & -1 & 0 & 0 & 1 & 0 & -1 & 0 & 1 & 0 & 1 & 1 & 1 & 1 & -1 & 3 & 15 & 360 \\
-1 & -1 & 0 & 1 & 1 & -1 & 1 & 0 & 1 & 1 & 1 & 1 & 0 & -1 & -1 & 3 & 15 & 360 \\
\bottomrule
\end{tabular}
\hspace{1mm}
\begin{tabular}{ccccccccccccccc|c||c|c}
\toprule
\multicolumn{18}{l}{Generator 8: Clique-Web-Generator}  \\
\midrule
-2 & -2 & -2 & 1 & 2 & -1 & -1 & 1 & 1 & -1 & 1 & 1 & 1 & 1 & 0 & 3 & 18 & 120 \\
-2 & -2 & -1 & 2 & 2 & -1 & -1 & 1 & 1 & -1 & 1 & 1 & 0 & 1 & -1 & 3 & 18 & 360 \\
-1 & -1 & -1 & -1 & 2 & -1 & -1 & -1 & 2 & -1 & -1 & 2 & 0 & 1 & 2 & 3 & 18 & 120 \\
-3 & -2 & -1 & 2 & 2 & -1 & -2 & 2 & 2 & 0 & 1 & 1 & 1 & 1 & -1 & 4 & 83 & 180 \\
-2 & -2 & -2 & 1 & 3 & -1 & -1 & 1 & 2 & -1 & 1 & 2 & 1 & 2 & -2 & 4 & 86 & 120 \\
-3 & -2 & 1 & 2 & 2 & -2 & 2 & 1 & 2 & 1 & 1 & 1 & 0 & -1 & -1 & 5 & 15 & 360 \\
-3 & -2 & 1 & 2 & 2 & -2 & 2 & 2 & 2 & 1 & 1 & 1 & -1 & -1 & -1 & 5 & 15 & 360 \\
-3 & -2 & 1 & 2 & 2 & -1 & 2 & 2 & 2 & 0 & 1 & 1 & -1 & -1 & -1 & 5 & 15 & 360 \\
-2 & -2 & -2 & 2 & 3 & -1 & -1 & 1 & 2 & -1 & 1 & 2 & 1 & 2 & -1 & 5 & 15 & 120 \\
-2 & -2 & -1 & 2 & 3 & -1 & -1 & 1 & 2 & -1 & 1 & 2 & 0 & 2 & -1 & 5 & 15 & 360 \\
-2 & -2 & -1 & 2 & 3 & -1 & -1 & 1 & 2 & -1 & 1 & 2 & 1 & 2 & -2 & 5 & 15 & 360 \\
-2 & -2 & -1 & 2 & 3 & -1 & -1 & 1 & 2 & 0 & 1 & 1 & 1 & 2 & -2 & 5 & 15 & 720 \\
-2 & -2 & 1 & 2 & 3 & -1 & 1 & 1 & 2 & 1 & 1 & 2 & -1 & -2 & -2 & 5 & 15 & 360 \\
\midrule
\multicolumn{18}{l}{Generator 9: Clique-Web-Generator} \\
\midrule
-2 & -2 & -2 & -2 & 3 & -1 & -1 & -1 & 2 & -1 & -1 & 2 & -1 & 2 & 2 & 3 & 15 & 30 \\
-3 & -3 & -3 & 3 & 5 & -1 & -1 & 1 & 2 & -1 & 1 & 2 & 1 & 2 & -2 & 6 & 15 & 120\\
-2 & -2 & -2 & -2 & 5 & -1 & -1 & -1 & 3 & -1 & -1 & 3 & -1 & 3 & 3 & 6 & 15 & 30\\
-5 & -5 & 3 & 3 & 3 & -3 & 2 & 2 & 2 & 2 & 2 & 2 & -1 & -1 & -1 & 7 & 73 & 60\\
-5 & -3 & 3 & 3 & 5 & -2 & 2 & 2 & 3 & 1 & 1 & 2 & -1 & -2 & -2 & 8 & 15 & 360\\
-3 & -3 & -3 & 5 & 5 & -1 & -1 & 2 & 2 & -1 & 2 & 2 & 2 & 2 & -3 & 8 & 15 & 60\\
-3 & -2 & -2 & 2 & 5 & -2 & -2 & 2 & 5 & -1 & 1 & 3 & 1 & 3 & -3 & 8 & 15 & 180\\
\midrule
\multicolumn{18}{l}{Generator 10: Parachute-Generator} \\
\midrule
-1 & -1 & 0 & 0 & 1 & 0 & -1 & 0 & 1 & 1 & -1 & 0 & 1 & 1 & -1 & 2 & 93 & 360\\
-1 & -1 & 0 & 0 & 1 & -1 & 0 & 1 & 0 & 1 & 0 & 1 & 1 & -1 & 1 & 3 & 15 & 720\\
-1 & -1 & 0 & 0 & 1 & -1 & 0 & 1 & 1 & 1 & 1 & 0 & -1 & 1 & 0 & 3 & 15 & 720\\
-1 & -1 & 0 & 0 & 1 & 0 & -1 & 0 & 1 & 1 & -1 & 1 & 1 & 0 & 1 & 3 & 15 & 720\\
-1 & -1 & 0 & 0 & 1 & 0 & -1 & 1 & 0 & 1 & -1 & 1 & 1 & 0 & 1 & 3 & 15 & 720\\
-1 & -1 & 0 & 0 & 1 & 0 & -1 & 1 & 1 & 1 & 0 & 1 & 1 & 0 & -1 & 3 & 15 & 720\\
\midrule
\multicolumn{18}{l}{Generator 11: Grishukhin-Generator} \\
\midrule
-1 & -1 & -1 & 0 & 1 & -1 & -1 & 0 & 1 & -1 & 1 & 0 & 1 & 0 & 1 & 2 & 19 & 90\\
-1 & -1 & -1 & 0 & 1 & -1 & -1 & 0 & 1 & 0 & -1 & 1 & 1 & 0 & 1 & 2 & 19 & 360\\
-2 & -2 & -1 & 1 & 2 & -1 & 0 & 1 & 1 & 0 & 1 & 1 & -1 & 1 & 0 & 3 & 87 & 360\\
-2 & -2 & 1 & 2 & 2 & -1 & 0 & 1 & 1 & 0 & 1 & 1 & -1 & -1 & -1 & 3 & 88 & 180\\
-2 & -2 & -1 & 2 & 2 & -1 & 0 & 1 & 1 & 0 & 1 & 1 & 1 & 1 & -1 & 4 & 15 & 180\\
-2 & -2 & 1 & 1 & 2 & -1 & 0 & 1 & 1 & 0 & 1 & 1 & 1 & -1 & 0 & 4 & 15 & 360\\
-2 & -2 & 1 & 1 & 2 & -1 & 0 & 1 & 1 & 1 & 0 & 1 & 1 & -1 & 0 & 4 & 15 & 720\\
-2 & -2 & 1 & 2 & 2 & -1 & 0 & 1 & 1 & 1 & 1 & 1 & -1 & 0 & -1 & 4 & 15 & 720\\
-2 & -1 & -1 & 2 & 2 & -1 & 0 & 1 & 1 & 1 & 0 & 0 & 1 & 1 & -1 & 4 & 15 & 360\\
-2 & -1 & -1 & 2 & 2 & -1 & 0 & 1 & 1 & 1 & 0 & 1 & 1 & 0 & -1 & 4 & 15 & 720\\
-2 & -1 & 1 & 1 & 1 & -1 & 2 & 2 & 2 & 0 & 0 & 1 & -1 & -1 & -1 & 4 & 15 & 360\\
-2 & -1 & 1 & 2 & 2 & 0 & 1 & 1 & 1 & -1 & 0 & 1 & -1 & 0 & -1 & 4 & 15 & 720\\
-2 & 0 & 1 & 1 & 1 & 1 & 2 & 2 & 2 & -1 & -1 & 0 & -1 & -1 & -1 & 4 & 15 & 360\\
-1 & -1 & -1 & 0 & 2 & -1 & -1 & 0 & 2 & -1 & 1 & 2 & 1 & 2 & -1 & 4 & 15 & 180\\
-1 & -1 & -1 & 0 & 2 & -1 & -1 & 0 & 2 & 0 & -1 & 2 & 1 & 1 & 1 & 4 & 15 & 360\\
-1 & -1 & -1 & 0 & 2 & -1 & 0 & 1 & 2 & 0 & 1 & 2 & -1 & 1 & -1 & 4 & 15 & 360\\
\bottomrule
\end{tabular}
\caption{
The 67 representatives for the 18720 facets of $\tcf_{6}$ as computed from the 11 generators of the facets of $\cut_{7}$ and the 28895 vertices of $\tcf_{6}$ (see Section~\ref{sect:compresults} and Tables~\ref{table:cutgenerators} and~\ref{table:TCFsixvertices}).  When we use the format $\sum_{1\leq i < j \leq 6}c_{ij}x_{ij} \leq c_{0}$ for the facet inducing inequalities of $\tcf_6$, columns (1)-(16) list the coefficients $c_{1,2},c_{1,3},\ldots,c_{5,6}$ followed by the constant $c_{0}$, column (17) is the total number of vertices from $\tcf_{6}$ solving it as an equation
and finally, column (18) is the orbit length under permutations. By ``new inequalities'' we mean that the following inequalities cannot be obtained as liftings from $\tcf_5$ (see Section~\ref{sect:lifting}).  
\label{table:TCFsixfacets}}
\end{table}

\begin{table}[p]\footnotesize \centering
\setlength{\tabcolsep}{0.9mm}
\begin{tabular}{rlp{2mm}ccccccccccccccccccccc}
\toprule
\multicolumn{24}{c}{\textbf{Generators for the cut polytope $\cut_7$}} \\
\midrule
\multicolumn{2}{r}{Name in \cite{dezalaurent_97}} && \multicolumn{21}{c}{Coefficients $c_{1,2},\ldots,c_{6,7}$}\\
\midrule
1. & $Q_7(1,1,-1,0,0,0,0)$ && 1&-1&0&0&0&0&-1&0&0&0&0&0&0&0&0&0&0&0&0&0&0\\
2. & $Q_7(1,1,1,-1,-1,0,0)$ && 1&1&-1&-1&0&0&1&-1&-1&0&0&-1&-1&0&0&1&0&0&0&0&0 \\
3. & $Q_7(2,1,1,-1,-1,-1,0)$ && 2&2&-2&-2&-2&0&1&-1&-1&-1&0&-1&-1&-1&0&1&1&0&1&0&0\\
4. & $Q_7(1,1,1,-1,-1,-1,-1)$ && 1&1&1&-1&-1&-1&1&1&-1&-1&-1&1&-1&-1&-1&-1&-1&-1&1&1&1\\
5. & $Q_7(2,2,1,-1,-1,-1,-1)$ && 4&2&-2&-2&-2&-2&2&-2&-2&-2&-2&-1&-1&-1&-1&1&1&1&1&1&1\\
6. & $Q_7(3,1,1,-1,-1,-1,-1)$ && 3&3&-3&-3&-3&-3&1&-1&-1&-1&-1&-1&-1&-1&-1&1&1&1&1&1&1\\
\midrule
7. & $\text{CW}^1_7(1,1,1,1,1,-1,-1)$ && 0&1&1&0&-1&-1&0&1&1&-1&-1&0&1&-1&-1&0&-1&-1&-1&-1&1\\
8. & $\text{CW}^1_7(2,2,1,1,-1,-1,-1)$ && 3&2&1&-2&-2&-2&1&2&-2&-2&-2&0&-1&-1&-1&-1&-1&-1&1&1&1\\
9. & $\text{CW}^1_7(3,2,2,-1,-1,-1,-1)$ && 5&5&-3&-3&-3&-3&3&-2&-2&-2&-2&-2&-2&-2&-2&1&1&1&1&1&1\\
\midrule
10. & $\text{Par}_7$ && -1&-1&0&-1&-1&0&1&0&1&0&-1&1&0&0&-1&-1&-1&-1&1&0&1 \\
\midrule
11. & $\text{Gr}_7$ && 1&1&1&-2&-1&0&1&1&-2&0&-1&1&-2&-1&0&-2&0&-1&1&1&-1\\
\bottomrule
\end{tabular}
\caption{ 
The 11 homogeneous generators for the $116\,764$  facets of $\cut_{7}$ under all switchings and permutations as in \cite{dezalaurent_97} p.~504 
and their 21 coefficients $c_{1,2},\ldots,c_{6,7}$ of $\sum_{1\leq i < j \leq 7}c_{ij}x_{ij} \leq 0$. Generators 1-6 are hypermetric ``in the cut sense'', i.e., the given $b$-vectors determine the $c_{ij}$ via $c_{ij}=b_i\cdot b_j$. Generators 7-9 are called clique-web inequalities (the vectors have a different meaning here). Generator 10 is a parachute inequality and generator 11 a Grishukhin inequality.
\label{table:cutgenerators}}
\vspace{7mm}
\setlength{\tabcolsep}{2mm}
\begin{tabular}{lll|l}
\toprule
\multicolumn{4}{c}{\textbf{Hypermetric facets of $\tcf_6$ and their corresponding $b$-vector}}  \\
\midrule
Generator 1 & $b=(1,1,0,0,0,0)$ & lifting of \emph{positivity inequality} &  $v_0$\\
& $b=(1,1,-1,0,0,0)$ & lifting of \emph{triangle inequality} &  $v_1$\\
\midrule
Generator 2 & $b=(1,1,1,-1,0,0)$ & lifting of \emph{tetrahedron inequality} \\
& $b=(1,1,1,-1,-1,0)$ & lifting of \emph{pentagonal inequality} &  $v_1$\\
\midrule
Generator 3 & $b=(1,1,1,1,-1,0)$ & lifting of \emph{pyramid inequality} \\
& $b=(2,1,1,-1,-1,0)$ & lifting of \emph{2-weighted pentagonal inequality} \\
& $b=(1,1,1,1,-2,0)$ & lifting of \emph{2-weighted pyramid inequality} \\
& $b=(2,1,1,-1,-1,-1)$ & &  $v_1$\\
& $b=(1,1,1,1,-1,-2)$ & &  $v_1$\\
\midrule
Generator 4 & $b=(1,1,1,1,-1,-1)$ & \\
\midrule
Generator 5 & $b=(2,1,1,1,-1,-1)$ & \\
& $b=(1,1,1,1,1,-2)$ & \\
& $b=(2,2,1,-1,-1,-1)$ & \\
& $b=(2,1,1,1,-1,-2)$ & \\
\midrule
Generator 6 & $b=(1,1,1,1,1,-1)$ & \\
& $b=(3,1,1,-1,-1,-1)$ & \\
& $b=(1,1,1,1,1,-3)$ & \\
\bottomrule
\end{tabular}
\caption{
The 17 representatives for the 858 hypermetric facets of $\tcf_{6}$ and their corresponding $b$-vectors (see Section~\ref{sect:basicfacts}). The list is in the same order as in Table~\ref{table:TCFsixfacets}. The last column indicates whether the respective facet contains one of the exposed vertices $v_0=(0,0,\dots,0)$ or $v_1=(1,1,\dots,1)$.
\label{table:hypermetricTCFsixfacets}}
\end{table}

\end{document}